\documentclass[11pt]{amsart}

\usepackage{amsmath,amsfonts,amssymb,mathabx,shuffle,latexsym}
\usepackage[usenames]{color}

\newtheorem{theorem}{Theorem}[section]
\newtheorem{lemma}[theorem]{Lemma}
\newtheorem{proposition}[theorem]{Proposition}
\newtheorem{corollary}[theorem]{Corollary}

\theoremstyle{definition}
\newtheorem{definition}[theorem]{Definition}

\theoremstyle{remark}
\newtheorem{remark}[theorem]{Remark}

\numberwithin{equation}{section}

\newcommand{\QSym}{\ensuremath{\mathrm{QSym}}}
\newcommand{\Poly}{\ensuremath{\mathrm{Poly}}}

\newcommand{\SSYT}{\ensuremath{\mathrm{SSYT}}}
\newcommand{\SYT}{\ensuremath{\mathrm{SYT}}}
\newcommand{\QYT}{\ensuremath{\mathrm{QYT}}}
\newcommand{\PD}{\ensuremath{\mathrm{PD}}}
\newcommand{\QPD}{\ensuremath{\mathrm{QPD}}}
\newcommand{\RD}{\ensuremath{\mathrm{R}}}

\newcommand{\wt}{\ensuremath{\mathrm{wt}}}
\newcommand{\inv}{\ensuremath{\mathrm{inv}}}
\newcommand{\width}{\ensuremath{\mathrm{width}}}

\newcommand{\sch}{\ensuremath{\mathfrak{S}}}
\newcommand{\sta}{\ensuremath{S}}

\newcommand{\Mono}{\ensuremath{\mathfrak{M}}}
\newcommand{\Fund}{\ensuremath{\mathfrak{F}}}

\newcommand{\stand}{\ensuremath{\mathrm{std}}}
\newcommand{\destand}{\ensuremath{\mathrm{dst}}}
\newcommand{\sit}{\ensuremath{\mathrm{sit}}}

\newcommand{\flatten}{\ensuremath{\mathrm{flat}}}
\newcommand{\bump}{\ensuremath{\mathrm{bump}}}

\newcommand{\QSS}{\ensuremath{\mathrm{QSS}}}
\newcommand{\ShS}{\ensuremath{\mathrm{SS}}}

\newcommand{\sgeq}{\ensuremath{\unrhd}}

\newcommand{\Des}{\ensuremath{\mathrm{Des}}}

\newcommand{\qshuffle}{\ensuremath{\squplus}}
\newcommand{\qslide}{\ensuremath{\squplus}}
\newcommand{\slide}{\ensuremath{\shuffle}}

\newlength\cellsize 

\savebox2{%
\begin{picture}(12,12)
\put(0,0){\line(1,0){12}}
\put(0,0){\line(0,1){12}}
\put(12,0){\line(0,1){12}}
\put(0,12){\line(1,0){12}}
\end{picture}}

\newcommand\cellify[1]{\def\thearg{#1}\def\nothing{}%
\ifx\thearg\nothing\vrule width0pt height\cellsize depth0pt%
  \else\hbox to 0pt{\usebox2\hss}\fi%
  \vbox to 12\unitlength{\vss\hbox to 12\unitlength{\hss$#1$\hss}\vss}}

\newcommand\tableau[1]{\vtop{\let\\=\cr
\setlength\baselineskip{-12000pt}
\setlength\lineskiplimit{12000pt}
\setlength\lineskip{0pt}
\halign{&\cellify{##}\cr#1\crcr}}}

\newcommand\elbow{
\begin{picture}(10,10)
\thicklines
\put(10,10){\oval(10,10)[bl]}
\put(0,0){\oval(10,10)[tr]}
\end{picture}}

\newcommand\upelb{
\begin{picture}(10,10)
\thicklines
\put(0,0){\oval(10,10)[tr]}
\end{picture}}

\newcommand\cross{
\begin{picture}(10,10)
\thicklines
\put(5,0){\line(0,1){10}}
\put(0,5){\line(1,0){10}}
\end{picture}}

\newcommand\gridify[1]{\vbox to 10\unitlength{\vss\hbox to 10\unitlength{\hss$_{#1}$\hss}\vss}}

\newcommand\pipes[1]{\vtop{\let\\=\cr
\setlength\baselineskip{-10000pt}
\setlength\lineskiplimit{10000pt}
\setlength\lineskip{0pt}
\halign{&\gridify{##}\cr#1\crcr}}}

\begin{document}


\title[Slide polynomials]{Schubert polynomials, slide polynomials, \\ Stanley symmetric functions and \\ quasi-Yamanouchi pipe dreams}  

\author[S. Assaf]{Sami Assaf}
\address{Department of Mathematics, University of Southern California, Los Angeles, CA 90089}
\email{shassaf@usc.edu}

\author[D. Searles]{Dominic Searles}
\address{Department of Mathematics, University of Southern California, Los Angeles, CA 90089}
\email{dsearles@usc.edu}


\date{March 11, 2016}



\begin{abstract}
  We introduce two new bases for polynomials that lift monomial and fundamental quasisymmetric functions to the full polynomial ring. By defining a new condition on pipe dreams, called quasi-Yamanouchi, we give a positive combinatorial rule for expanding Schubert polynomials into these new bases that parallels the expansion of Schur functions into fundamental quasisymmetric functions. As a result, we obtain a refinement of the stable limits of Schubert polynomials to Stanley symmetric functions. We also give combinatorial rules for the positive structure constants of these bases that generalize the quasi-shuffle product and shuffle product, respectively. We use this to give a Littlewood--Richardson rule for expanding a product of Schubert polynomials into fundamental slide polynomials and to give formulas for products of Stanley symmetric functions in terms of Schubert structure constants.
\end{abstract}

\maketitle

%
\section{Introduction}
%
\label{sec:introduction}

The Schubert polynomials give explicit polynomial representatives for the Schubert classes in the cohomology ring of the complete flag variety, with the goal of facilitating computations of intersection numbers. Lascoux and Sch\"utzenberger \cite{LS82} first defined Schubert polynomials indexed by permutations in terms of divided difference operators, and later Billey, Jockusch and Stanley \cite{BJS93} and Fomin and Stanley \cite{FS94} gave direct monomial expansions. Bergeron and Billey \cite{BB93} reformulated this again to give a beautiful combinatorial definition of Schubert polynomials as generating functions for $RC$-graphs, often called \emph{pipe dreams}. However, even armed with these elegant formulations, the longstanding problem of giving a positive combinatorial formula for the structure constants of Schubert polynomials remains open in all but a few special cases.

In this paper, we introduce a new tool to aid in the study of Schubert polynomials. We define two new families of polynomials we call the \emph{monomial slide polynomials} and \emph{fundamental slide polynomials}. Both monomial and fundamental slide polynomials are combinatorially indexed by weak compositions, and both families form a basis of the polynomial ring. Moreover, the Schubert polynomials expand positively into the fundamental slide basis, which in turn expands positively into the monomial slide basis. While there are other bases that refine Schubert polynomials, most notably key polynomials \cite{Dem74,LS90}, ours has two main properties that make it a compelling addition to the theory of Schubert calculus. First, our polynomials exhibit a similar stability to that of Schubert polynomials, and so they facilitate a deeper understanding of the stable limit of Schubert polynomials, which, as originally shown by Macdonald \cite{Mac91}, are Stanley symmetric functions \cite{Sta84}. Second, and in sharp contrast to key polynomials, our bases themselves have \emph{positive} structure constants, and so our Littlewood-Richardson rule for the fundamental slide expansion of a product of Schubert polynomials takes us one step closer to giving a combinatorial formula for Schubert structure constants.

To motivate our new bases, let us first recall a special case in which the Schubert problem is solved explicitly, that of the Grassmannian partial flag variety. In this case, Schubert polynomials are nothing more than Schur polynomials, which form a well-studied basis for symmetric polynomials, that is, polynomials invariant under any permutation of the variables. Schur polynomials have a beautiful combinatorial definition as the generating functions of semistandard Young tableaux, and the original Littlewood--Richardson rule gives an elegant combinatorial formula for the Schur structure constants as the number of so-called \emph{Yamanouchi} tableaux, which are semistandard tableaux satisfying certain additional conditions. This rule has many reformulations and many beautiful proofs, yet so far none of these has been lifted to the general polynomial setting.

As an intermediate step to this lift, we consider instead the ring of quasisymmetric polynomials, that is, polynomials invariant under certain permutations of the variables. Gessel \cite{Ges84} defined the fundamental basis for quasisymmetric polynomials, and showed that the Schur polynomials may be written as the generating function of standard Young tableaux when monomials are replaced with fundamental quasisymmetric polynomials. While the number of semistandard Young tableaux depends on the number of variables used, the number of standard Young tableaux is independent of the number of variables. Therefore Gessel's expansion of Schur polynomials is significantly more compact, and makes computations far more efficient. However, even this expansion can be improved upon since, when the number of variables is small enough, the contribution of certain standard Young tableaux is zero. To resolve this, we introduce \emph{quasi-Yamanouchi tableaux} so that the fundamental quasisymmetric expansion of a Schur polynomial is precisely given by summing over quasi-Yamanouchi tableaux. This theory is developed in Section~\ref{sec:poly-qyam} after a review of Schur polynomials and quasisymmetric polynomials in Sections~\ref{sec:poly-sym} and \ref{sec:poly-qsym}, respectively.

The fundamental slide polynomials, indexed by weak compositions, are a lifting of the fundamental quasisymmetric polynomials, and the fundamental slide expansion of Schubert polynomials is precisely given by summing over \emph{quasi-Yamanouchi pipe dreams}. Just as quasi-Yamanouchi tableaux correspond to a subset of standard Young tableaux, quasi-Yamanouchi pipe dreams correspond to a subset of reduced decompositions for the indexing permutation. This gives a significantly more compact expansion for Schubert polynomials, which makes calculations far more tractable. We define slide polynomials in Section~\ref{sec:schub-slide} after reviewing Schubert polynomials in Section~\ref{sec:schub-pipe}. We extend the quasi-Yamanouchi condition to pipe dreams in Section~\ref{sec:schub-quasi}, and use it to give the fundamental slide polynomials expansion of Schubert polynomials.

One can take the \emph{stable limit} of a Schubert polynomial by embedding a permutation of $n$ into the larger symmetric group on $m+n$ and fixing the first $m$ positions. Macdonald \cite{Mac91} showed that these limits are well-defined and are exactly the Stanley symmetric functions \cite{Sta84}. The slide polynomials also have well-defined stable limits, with the monomial slide polynomials converging to monomial quasisymmetric functions and the fundamental slide polynomials converging to fundamental quasisymmetric functions. In the process, the correspondence between quasi-Yamanouchi pipe dreams and reduced decompositions becomes a bijection, and the convergence of Schubert polynomials to Stanley symmetric functions becomes clear. We give a refined notion of this stability and when it occurs. We show in Section~\ref{sec:stable-vars} that trivially increasing the number of variables leaves our functions unchanged, just as in the Schubert setting. In Section~\ref{sec:stable-limit}, we recall Stanley symmetric functions and derive the stable limits of the slide polynomials. In Section~\ref{sec:stable-refine}, we use this to understand the convergence of Schubert polynomials to Stanley symmetric functions by considering the stability of fundamental slide expansion of Schubert polynomials.

Returning to the motivating open problem of computing structure constants, in Section~\ref{sec:struct-quasi} we give a positive combinatorial rule for the structure constants of the monomial slide polynomials by generalizing the quasi-shuffle product of Hoffman \cite{Hof00}. We follow this in Section~\ref{sec:struct-slide} by giving a positive combinatorial rule for the structure constants of the fundamental slide polynomials, by means of a generalization of the shuffle product of Eilenberg and Mac Lane \cite{EM53} to weak compositions that we call the \emph{slide product}. Finally, in Section~\ref{sec:struct-schub}, we apply the slide product to give a positive Littlewood--Richardson rule for the fundamental slide expansion of a product of Schubert polynomials. By taking the stable limit, we tighten a theorem of Li \cite{Li14} stating that the product of Schubert polynomials stabilizes, and, consequently, that the product of Stanley symmetric functions can be expressed in terms of Schubert structure constants.

%
\section{Schur polynomials}
%
\label{sec:poly}

\subsection{Semistandard Young tableaux}
\label{sec:poly-sym}

We adopt notation and terminology for symmetric polynomials from \cite{Mac95}, beginning with $\Lambda_n$, the ring of polynomials in $\mathbb{Z}[x_1,\ldots,x_n]$ that are invariant under any permutation of the variables. That is, a polynomial $f \in \mathbb{Z}[x_1,\ldots,x_n]$ is \emph{symmetric} if for every (strong) composition $\alpha = (\alpha_1,\ldots,\alpha_{\ell})$, with $\ell \leq n$ and $\alpha_i>0$ for all $i$, we have
\begin{equation}
  [x_{i_1}^{\alpha_1} \cdots x_{i_{\ell}}^{\alpha_{\ell}} \mid f] = [x_{j_1}^{\alpha_1} \cdots x_{j_{\ell}}^{\alpha_{\ell}} \mid  f] 
  \label{e:symmetric}
\end{equation}
for any two sequences $(i_1,\ldots,i_{\ell}), (j_1,\ldots,j_{\ell})$ of distinct elements of $[n] = \{1,2,\ldots,n\}$, where $[x^a \mid f]$ means the coefficient of $x^a$ in $f$.

The dimension of $\Lambda_n$ as a $\mathbb{Z}$-module is the number of integer partitions of length at most $n$. A \emph{partition} is sequence $(\lambda_1 \geq \cdots \geq \lambda_\ell > 0)$ of nonnegative integers. The \emph{length of $\lambda$}, denoted by $\ell(\lambda)$, is the number of (nonzero) parts. The \emph{size of $\lambda$}, denoted by $|\lambda|$, is the sum of the parts. We draw the \emph{diagram} of a partition $\lambda$ in French notation as the set of points $(i,j)$ in the $\mathbb{Z} \times\mathbb{Z}$ lattice such that $1 \leq i \leq \lambda_j$; see Figure~\ref{fig:5441}.

\begin{figure}[ht]
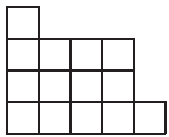

  \begin{displaymath}
    \tableau{ \ \\
      \ & \ & \ & \ \\
      \ & \ & \ & \ \\
      \ & \ & \ & \ & \ } 
  \end{displaymath}
  \caption{\label{fig:5441} The diagram for $(5,4,4,1)$.}
\end{figure}

The ring $\Lambda_n$ is graded by degree, namely
\begin{equation}
  \Lambda_n = \bigoplus_{k \geq 0} \Lambda_n^k
  \label{e:sym-graded}
\end{equation}
where $\Lambda_n^k$ consists of zero together with those symmetric polynomials homogeneous of degree $k$. As a $\mathbb{Z}$-module, $\Lambda_n^k$ has dimension equal to the number of partitions of length at most $n$ and size $k$.

By taking the inverse limit with respect to the homomorphisms $\rho_{m,n}^k : \Lambda_m^k \rightarrow \Lambda_n^k$ that specialize the variables $x_{n+1},\ldots,x_m$ to zero, we form the symmetric \emph{functions} homogeneous of degree $k$,
\begin{equation}
  \Lambda^k = \lim_{\infty\leftarrow n} \Lambda_n^k.
  \label{e:sym-limit}
\end{equation}
And, of course, we have the full ring of symmetric functions $\Lambda = \bigoplus_{k \geq0} \Lambda^k$. One may (and many do) study the symmetric polynomial ring $\Lambda_n$ by first understanding the symmetric function ring $\Lambda$ and then specializing trailing variables to zero. However, in this paper we maintain that by studying symmetric polynomials and the ways in which they are different from symmetric functions, we gain additional insights that will allow us to lift powerful ideas from the symmetric setting to arbitrary polynomials.

There are many nice bases for $\Lambda_n^k$ as beautifully exposited in \cite{Mac95}. For our current purposes, we are primarily interested in the most interesting basis, the Schur basis denoted by $\{s_{\lambda}\}$. Originally defined as a ratio of determinants, we instead give the combinatorial definition of a Schur polynomial as the generating function of semistandard Young tableaux.

A \emph{semistandard Young tableau of shape $\lambda$} is a map $T : \lambda \rightarrow \mathbb{N}$ such that
\begin{itemize}
\item $T(c) \leq T(d)$ if $c,d$ are cells in the same row of $\lambda$ with $c$ left of $d$, and 
\item $T(c) < T(d)$ if $c,d$ are cells in the same column of $\lambda$ with $c$ below $d$.
\end{itemize}
Let $\SSYT_n(\lambda)$ denote the set of semistandard Young tableaux with $T(\lambda) \subseteq [n]$. For example, the semistandard Young tableaux of shape $(3,2)$ with image in $[3]$ are given in Figure~\ref{fig:SSYT}.

\begin{figure}[ht]
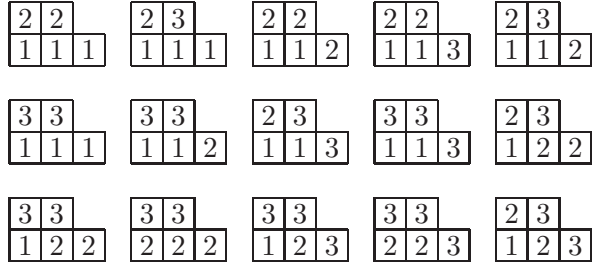

  \begin{displaymath}
    \begin{array}{ccccc}
      \tableau{ 2 & 2 \\ 1 & 1 & 1} &
      \tableau{ 2 & 3 \\ 1 & 1 & 1} &
      \tableau{ 2 & 2 \\ 1 & 1 & 2} &
      \tableau{ 2 & 2 \\ 1 & 1 & 3} &
      \tableau{ 2 & 3 \\ 1 & 1 & 2} \\ \\
      \tableau{ 3 & 3 \\ 1 & 1 & 1} &
      \tableau{ 3 & 3 \\ 1 & 1 & 2} &
      \tableau{ 2 & 3 \\ 1 & 1 & 3} &
      \tableau{ 3 & 3 \\ 1 & 1 & 3} &
      \tableau{ 2 & 3 \\ 1 & 2 & 2} \\ \\
      \tableau{ 3 & 3 \\ 1 & 2 & 2} &
      \tableau{ 3 & 3 \\ 2 & 2 & 2} &
      \tableau{ 3 & 3 \\ 1 & 2 & 3} &
      \tableau{ 3 & 3 \\ 2 & 2 & 3} &
      \tableau{ 2 & 3 \\ 1 & 2 & 3}
    \end{array}
  \end{displaymath}
  \caption{\label{fig:SSYT} The $15$ elements of $\SSYT_3(3,2)$.}
\end{figure}

A \emph{weak composition} is a sequence of nonnegative integers. To each $T \in \SSYT_n$, we associate the weak composition $\wt(T)$ whose $i$th component is equal to the number of occurrences of $i$ in $T$. For example, the weights of the first column of tableaux in Figure~\ref{fig:SSYT} are $(3,2,0),(3,0,2),(1,2,2)$, from top to bottom.

\begin{definition}
  The \emph{Schur polynomial} $s_{\lambda}(x_1,\ldots,x_n)$ is given by
  \begin{equation}
    s_{\lambda}(x_1,\ldots,x_n) = \sum_{T \in \SSYT_n(\lambda)} x^{\wt(T)},
    \label{e:schur}
  \end{equation}
  where $x^{a}$ is the monomial $x_1^{a_1} \cdots x_n^{a_n}$.
  \label{def:schur}
\end{definition}

For example, from Figure~\ref{fig:SSYT} we can compute
\begin{eqnarray}
  s_{(3,2)}(x_1,x_2,x_3) & = &  x_1^3 x_2^2  + x_1^3 x_3^2 +  x_1^2 x_2^3  +   2 x_1^2 x_2^2 x_3  +   2 x_1^2 x_2 x_3^2  +   x_1^2 x_3^3 \\\label{e:s32_x}
  & & + x_1^3 x_2 x_3 + x_1 x_2^3 x_3  +   2 x_1 x_2^2 x_3^2  +   x_1 x_2 x_3^3  +   x_2^3 x_3^2  +   x_2^2 x_3^3. \nonumber
\end{eqnarray}
Had we chosen to compute $s_{(3,2)}(x_1,\ldots,x_4)$ instead, we would have summed over the $60$ elements of $\SSYT_4(3,2)$. 

Letting $n \rightarrow \infty$ gives the Schur functions, which are well-defined both by the unbounded version of \eqref{e:schur} and by the fact that $\rho_{n+1,n}(s_{\lambda}(x_1,\ldots,x_{n+1})) = s_{\lambda}(x_1,\ldots,x_n)$. Therefore, while \eqref{e:schur} gives a beautiful \emph{combinatorial} definition for $s_{\lambda}$, this formula quickly becomes intractable.

\subsection{Quasisymmetric polynomials}
\label{sec:poly-qsym}

To facilitate a tractable expression for Schur polynomials, we consider the larger ring of quasisymmetric polynomials, denoted by $\QSym_n$. A polynomial $f \in \mathbb{Z}[x_1,\ldots,x_n]$ is \emph{quasisymmetric} if for every (strong) composition $\alpha = (\alpha_1,\ldots,\alpha_{\ell})$ with $\ell \leq n$, we have
\begin{equation}
  [x_{i_1}^{\alpha_1} \cdots x_{i_{\ell}}^{\alpha_{\ell}} \mid  f] = [x_{j_1}^{\alpha_1} \cdots x_{j_{\ell}}^{\alpha_{\ell}} \mid  f] 
  \label{e:quasisymmetric}
\end{equation}
for any two sequences $1 \leq i_1< \cdots < i_{\ell} \leq n$ and $1 \leq j_1< \cdots < j_{\ell} \leq n$. Clearly a symmetric polynomial is also quasisymmetric, so $\Lambda_n \subset \QSym_n$.

Like $\Lambda_n$, the ring $\QSym_n$ is graded by degree, namely
\begin{equation}
  \QSym_n = \bigoplus_{k \geq 0} \QSym_n^k
  \label{e:qsym-graded}
\end{equation}
where $\QSym_n^k$ consists of zero together with those quasisymmetric polynomials homogeneous of degree $k$. As a $\mathbb{Z}$-module, $\QSym_n^k$ has dimension equal to the number of (strong) compositions of length at most $n$ and size $k$, where size is again defined to be the sum of the parts.

As with the symmetric case, we can consider quasisymmetric \emph{functions} as the inverse limit of their polynomial counterparts using the same specialization homomorphisms $\rho_{m,n}^k$. However, as our goal remains to study polynomials, we focus primarily on the polynomial setting.

There are many nice bases for $\QSym_n^k$. For our current purposes, we are fundamentally interested in the fundamental basis defined by Gessel in his study of $P$-partitions \cite{Ges84}. To define this, though, it is convenient first to define the \emph{monomial quasisymmetric polynomials}, denoted by $M_{\alpha}$. For $\alpha = (\alpha_1,\ldots,\alpha_{\ell})$, we have
\begin{equation}
  M_{\alpha}(x_1,\cdots,x_n) = \sum_{1 \leq i_1 < \cdots < i_{\ell} \leq n} x_{i_1}^{\alpha_1} \cdots x_{i_{\ell}}^{\alpha_{\ell}}.
  \label{e:monomial}
\end{equation}
For example, $M_{(2,3)}(x_1,x_2,x_3) = x_1^2 x_2^3 + x_1^2 x_3^3 + x_2^2 x_3^3$.

Given two compositions $\alpha$ and $\beta$ of the same size, say that $\beta$ \emph{refines} $\alpha$ if there exist indices $i_1 < \cdots < i_{\ell}$ such that $\beta_{i_j+1} + \cdots + \beta_{i_{j+1}} = \alpha_{j+1}$. For example, $(1,2,2)$ refines $(3,2)$ but not $(2,3)$.

\begin{definition}[\cite{Ges84}]
  The \emph{fundamental quasisymmetric polynomial} $F_{\alpha}(x_1,\ldots,x_n)$ is given by
  \begin{equation}
    F_{\alpha}(x_1,\ldots,x_n) = \sum_{\beta \ \mathrm{refines} \ \alpha} M_{\beta}(x_1,\ldots,x_n) .
    \label{e:fundamental}
  \end{equation}
  \label{def:fundamental}
\end{definition}

For example, we have
\begin{eqnarray}
  F_{(2,3)}(x_1,x_2,x_3)  & = &  M_{(2,3)}(x_1,x_2,x_3) + M_{(2,2,1)}(x_1,x_2,x_3)\\ \label{e:F23}
  & &  + M_{(2,1,2)}(x_1,x_2,x_3) + M_{(1,1,3)}(x_1,x_2,x_3). \nonumber
\end{eqnarray}
Note that there are additional compositions that refine $(2,3)$ that do not appear as indices on the right hand side, e.g. $(1,1,2,1)$, because their length exceeds the number of variables. 

The fundamental basis gives a more compact expansion for Schur polynomials in terms of standard Young tableaux. A \emph{standard Young tableau} is a semistandard Young tableau $T : \lambda \stackrel{\sim}{\rightarrow} [k]$, where $k = |\lambda|$. Let $\SYT(\lambda)$ denote the set of standard Young tableaux. Note that unlike semistandard tableaux, the set of standard tableaux is independent of $n$. For example, the standard Young tableaux of shape $(3,2)$ are given in Figure~\ref{fig:SYT}.

\begin{figure}[ht]
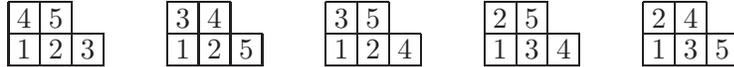

  \begin{displaymath}
    \begin{array}{c@{\hskip 2\cellsize}c@{\hskip 2\cellsize}c@{\hskip 2\cellsize}c@{\hskip 2\cellsize}c}
      \tableau{ 4 & 5 \\ 1 & 2 & 3 } &
      \tableau{ 3 & 4 \\ 1 & 2 & 5 } &
      \tableau{ 3 & 5 \\ 1 & 2 & 4 } &
      \tableau{ 2 & 5 \\ 1 & 3 & 4 } &
      \tableau{ 2 & 4 \\ 1 & 3 & 5 } 
    \end{array}
  \end{displaymath}
  \caption{\label{fig:SYT} The $5$ elements of $\SYT(3,2)$.}
\end{figure}

To each standard Young tableau $T$, we associate the \emph{descent composition}, denoted by $\Des(T)$, obtained by taking lengths of successive increasing runs of the entries by reading $1$ to $k$ in order, and beginning a new run whenever $i+1$ appears weakly left of $i$. For example, the descent compositions of the tableaux in Figure~\ref{fig:SYT} are, respectively, $(3,2), (2,3), (2,2,1), (1,3,1), (1,2,2)$.

\begin{theorem}[\cite{Ges84}]
  The Schur polynomial $s_{\lambda}(x_1,\ldots,x_n)$ is given by
  \begin{equation}
    s_{\lambda}(x_1,\ldots,x_n) = \sum_{T \in \SYT(\lambda)} F_{\Des(T)}(x_1,\ldots,x_n),
    \label{e:gessel}
  \end{equation}
  \label{thm:gessel}
\end{theorem}

For example, from Figure~\ref{fig:SYT} we can compute
\begin{eqnarray}
  s_{(3,2)}(x_1,x_2,x_3) & = & F_{(3,2)}(x_1,x_2,x_3) + F_{(2,3)}(x_1,x_2,x_3) + F_{(2,2,1)}(x_1,x_2,x_3) \\\label{e:s32_F}
  & & + F_{(1,3,1)}(x_1,x_2,x_3) + F_{(1,2,2)}(x_1,x_2,x_3) \nonumber
\end{eqnarray}
Whereas the number of terms in the monomial expansion of $s_{\lambda}$ given by \eqref{e:schur} increases as the number of variables increases, the number of terms in the fundamental expansion of $s_{\lambda}$ given by \eqref{e:gessel} is independent of the number of variables. Even for our small example of $s_{(3,2)}(x_1,x_2,x_3)$, the improvement of \eqref{e:s32_F} over \eqref{e:s32_x} is considerable. Taking inverse limits, the expansions in \eqref{e:gessel} are \emph{finite}, an infinite improvement over the monomial expansion.

While generally more compact, the monomial expansion \eqref{e:schur} does not always have more terms than \eqref{e:gessel} since some of the terms on the right hand side of \eqref{e:gessel} can be zero. For example, we have the following expansions for $s_{(3,2)}$ in two variables,
\begin{eqnarray}
  s_{(3,2)}(x_1,x_2) & = & x_1^3 x_2^2 + x_1^2 x_2^3 \\ \label{e:s32_F_2vars}
  & = & F_{(3,2)}(x_1,x_2) + F_{(2,3)}(x_1,x_2) \\ 
  & & + F_{(2,2,1)}(x_1,x_2) + F_{(1,3,1)}(x_1,x_2) + F_{(1,2,2)}(x_1,x_2).\nonumber
\end{eqnarray}

Note that the latter three terms in the latter expansion \eqref{e:s32_F_2vars} are, in fact, zero. This points to a missing concept in the theory that allows one to avoid writing out unnecessary terms.

\subsection{Quasi-Yamanouchi tableaux}
\label{sec:poly-qyam}

In order to avoid unnecessary terms and to complete a missing concept in Schur polynomial expansions, we introduce the notion of \emph{quasi-Yamanouchi Young tableaux}.

\begin{definition}
  A semistandard Young tableau is \emph{quasi-Yamanouchi} if for all $i>1$, the leftmost occurrence of $i$ lies weakly left of some $i-1$. Let $\QYT_n(\lambda)$ denote the set of quasi-Yamanouchi tableaux with image in $[n]$.
  \label{def:quasi-Yam-tab}
\end{definition}

For example, the quasi-Yamanouchi tableaux of shape $(3,2)$ with image in $[3]$ are given in Figure~\ref{fig:LYT}.

\begin{figure}[ht]
  \begin{displaymath}
    \begin{array}{c@{\hskip 2\cellsize}c@{\hskip 2\cellsize}c@{\hskip 2\cellsize}c@{\hskip 2\cellsize}c}
      \tableau{ 2 & 2 \\ 1 & 1 & 1 } &
      \tableau{ 2 & 2 \\ 1 & 1 & 2 } &
      \tableau{ 2 & 3 \\ 1 & 1 & 2 } &
      \tableau{ 2 & 3 \\ 1 & 2 & 2 } &
      \tableau{ 2 & 3 \\ 1 & 2 & 3 } 
    \end{array}
  \end{displaymath}
  \caption{\label{fig:LYT} The $5$ elements of $\QYT_3(3,2)$.}
\end{figure}

Note that if $i$ occurs in $T$ for some $i>1$, then $i-1$ must also occur in $T$. In particular, the weight of a quasi-Yamanouchi tableau is a strong composition. This implies that the number of quasi-Yamanouchi tableaux is bounded as $n$ grows. In fact, a stronger statement is true.

\begin{definition}
  Define the \emph{destandardization of $T$}, denoted by $\destand(T)$, to be the tableau constructed as follows. If the leftmost $i$ lies strictly right of the rightmost $i-1$, then decrement every $i$ to $i-1$. Repeat until no $i$ satisfies the condition.
  \label{def:T-destand}  
\end{definition}

For example, see Figure~\ref{fig:T-destand}.

\begin{figure}[ht]
  \begin{displaymath}
    \tableau{8 \\ 5 & 9 & 9 & 9 \\ 3 & 3 & 5 & 8 \\ 1 & 2 & 2 & 3 & 6}
    \hspace{1\cellsize} \stackrel{\destand}{\longrightarrow} \hspace{1\cellsize}
    \tableau{4 \\ 3 & 5 & 5 & 5 \\ 2 & 2 & 3 & 4 \\ 1 & 1 & 1 & 2 & 3}
  \end{displaymath}
  \caption{\label{fig:T-destand}An example of destandardization of a semistandard Young tableau.}
\end{figure}

\begin{lemma}
  The destandardization map is well-defined and satisfies the following
  \begin{enumerate}
  \item for $T \in \SSYT_n(\lambda)$, $\destand(T) \in \QYT_n(\lambda)$;
  \item for $T \in \SSYT_n(\lambda)$, $\destand(T)=T$ if and only if $T \in \QYT_n(\lambda)$;
  \item $\destand:\SSYT_n(\lambda) \rightarrow \QYT_n(\lambda)$ is surjective; and
  \item $\destand:\SSYT_n(\lambda) \rightarrow \QYT_n(\lambda)$ is injective if and only if $n \leq \ell(\lambda)$.
  \end{enumerate}
  \label{lem:T-destand}
\end{lemma}

\begin{proof}
  The process of destandardization terminates only if the quasi-Yamanouchi condition is satisfied, proving (1) and (2), and property (3) follows from (2). For property (4), both sets are empty if $n<\ell(\lambda)$, and when $n=\ell(\lambda)$, the first column of each semistandard tableaux contains $1,\ldots,\ell(\lambda)$, thus satisfying the quasi-Yamanouchi condition. For $n>\ell(\lambda)$, the filling with $i+1$ in every cell in row $i$ is not quasi-Yamanouchi. Hence the map is not injective.
\end{proof}

Our main purpose in introducing these new objects is obtain the following precise expansion for Schur polynomials.

\begin{theorem}
  The Schur polynomial $s_{\lambda}(x_1,\ldots,x_n)$ is given by
  \begin{equation}
    s_{\lambda}(x_1,\ldots,x_n) = \sum_{T \in \QYT_n(\lambda)} F_{\wt(T)}(x_1,\ldots,x_n),
    \label{e:left_schur}
  \end{equation}
  where all terms on the right hand side are nonzero.
  \label{thm:left_schur}
\end{theorem}

\begin{proof}
  Note that if $\destand(S)=T$, then $\wt(S)$ refines $\wt(T)$ since $T$ is obtained by changing \emph{all} $i$'s to $i-1$'s. Conversely, we claim that given $T \in \QYT_n(\lambda)$, for every weak composition $b$ of length $n$ such that $b$ with $0$ parts removed refines $\wt(T)$ as (strong) compositions, there is a unique $S \in \SSYT_n(\lambda)$ with $\wt(S) = b$ such that $\destand(S) = T$. From the claim, for $T \in \QYT_n(\lambda)$, we have
  \begin{displaymath}
    \sum_{S \in \destand^{-1}(T)} x^{\wt(S)} = F_{\wt(T)}(x_1,\ldots,x_n).
  \end{displaymath}
  The theorem follows from this and Lemma~\ref{lem:T-destand}.
  
  To construct $S$ from $b$ and $T$, for $j = \ell(\wt(T)),\ldots,1$, if $\wt(T)_{j} = b_{i_{j-1} + 1} + \cdots + b_{i_{j}}$, then, from left to right, change each of the first $b_{i_{j-1} + 1}$ $j$'s to $i_{j-1} + 1$, the next $b_{i_{j-1} + 2}$ $j$'s to $i_{j-1} + 2$, and so on. Existence is proved, and uniqueness follows from the lack of choice at every step. 
\end{proof}

For example, from Figure~\ref{fig:LYT} we can compute
\begin{equation}
  s_{(3,2)}(x_1,x_2) = F_{(3,2)}(x_1,x_2) + F_{(2,3)}(x_1,x_2).
\end{equation}

%
\section{Schubert polynomials}
%
\label{sec:schub}

\subsection{Pipe dreams}
\label{sec:schub-pipe}

We lift our attention now to the full polynomial ring in $n$ variables, $\Poly_n = \mathbb{Z}[x_1,\ldots,x_n]$, which contains both quasisymmetric polynomials and symmetric polynomials as subrings. The polynomial ring $\Poly_n$ is graded by degree, namely
\begin{equation}
  \Poly_n = \bigoplus_{k \geq 0} \Poly_n^k
  \label{e:poly-graded}
\end{equation}
where $\Poly_n^k$ consists of zero together with those polynomials homogeneous of degree $k$, and, of course, we have $\Lambda_n^k \subset \QSym_n^k \subset \Poly_n^k$. As a $\mathbb{Z}$-module, $\Poly_n^k$ has dimension equal to the number of weak compositions of length at most $n$ and size $k$, where size is again defined to be the sum of the parts.

Given a permutation $w \in \mathcal{S}_{\infty}$, written in one-line notation, say that a pair $(i,j)$ with $i<j$ is an \emph{inversion} of $w$ if $w_i > w_j$. Define the \emph{Lehmer code} $L(w)$ of a permutation $w$ to be the weak composition whose $i$th term is the number of indices $j$ for which $(i,j)$ is an inversion. For example, $L(146235) = (0,2,3,0,0,0)$. This defines a bijection between weak compositions and permutations. Say that $i$ is a \emph{descent} of $w$ if $w_i > w_{i+1}$. Using this bijection, an alternative indexing set for a basis of $\Poly_n^k$ is given by permutations $w \in \mathcal{S}_{\infty}$ with no descents beyond position $n$ and exactly $k$ inversions.

Schubert polynomials, denoted by $\sch_w$, originally defined by Lascoux and Sch\"{u}tzenberger \cite{LS82}, form an important $\mathbb{Z}$-basis for $\Poly_n^k$. Lascoux and Sch\"{u}tzenberger showed that the Schubert polynomials represent Schubert classes in the cohomology of the flag manifold. Originally defined in terms of divided difference operators, we instead give the combinatorial definition as the generating function of reduced pipe dreams \cite{BJS93,FS94}. For more on Schubert polynomials, see \cite{Mac91}.

Consistent with our treatment of tableaux, we adopt the French notation for pipe dreams as well. A \emph{(reduced) pipe dream} is a tiling of the first quadrant of $\mathbb{Z} \times \mathbb{Z}$ with \emph{elbows} $\elbow$ and finitely many \emph{crosses} $\cross$ such that no two lines, or \emph{pipes}, cross more than once. The \emph{shape} of a pipe dream is the permutation of $\mathcal{S}_{\infty}$ obtained by following the pipes from the $y$-axis to the $x$-axis. 

\begin{figure}[ht]
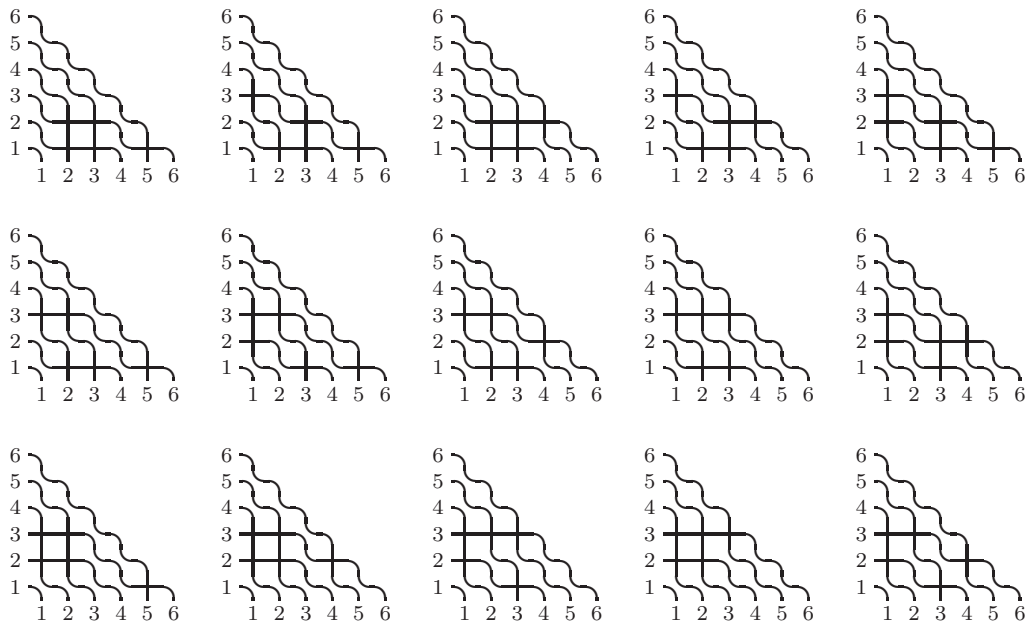

\begin{displaymath}
  \begin{array}{ccccc}
    \pipes{%
      6 & \upelb \\
      5 & \elbow & \upelb \\
      4 & \elbow & \elbow & \upelb \\
      3 & \elbow & \elbow & \elbow & \upelb \\
      2 & \elbow & \cross & \cross & \elbow & \upelb \\
      1 & \elbow & \cross & \cross & \elbow & \cross & \upelb \\
      & 1 & 2 & 3 & 4 & 5 & 6} &
    \pipes{%
      6 & \upelb \\
      5 & \elbow & \upelb \\
      4 & \elbow & \elbow & \upelb \\
      3 & \cross & \elbow & \elbow & \upelb \\
      2 & \elbow & \elbow & \cross & \elbow & \upelb \\
      1 & \elbow & \cross & \cross & \elbow & \cross & \upelb \\
      & 1 & 2 & 3 & 4 & 5 & 6} &
    \pipes{%
      6 & \upelb \\
      5 & \elbow & \upelb \\
      4 & \elbow & \elbow & \upelb \\
      3 & \elbow & \elbow & \elbow & \upelb \\
      2 & \elbow & \cross & \cross & \cross & \upelb \\
      1 & \elbow & \cross & \cross & \elbow & \elbow & \upelb \\
      & 1 & 2 & 3 & 4 & 5 & 6} &
    \pipes{%
      6 & \upelb \\
      5 & \elbow & \upelb \\
      4 & \elbow & \elbow & \upelb \\
      3 & \cross & \elbow & \elbow & \upelb \\
      2 & \elbow & \elbow & \cross & \cross & \upelb \\
      1 & \elbow & \cross & \cross & \elbow & \elbow & \upelb \\
      & 1 & 2 & 3 & 4 & 5 & 6} &
    \pipes{%
      6 & \upelb \\
      5 & \elbow & \upelb \\
      4 & \elbow & \elbow & \upelb \\
      3 & \cross & \elbow & \elbow & \upelb \\
      2 & \cross & \elbow & \cross & \elbow & \upelb \\
      1 & \elbow & \elbow & \cross & \elbow & \cross & \upelb \\
      & 1 & 2 & 3 & 4 & 5 & 6}  \\ \\
    \pipes{%
      6 & \upelb \\
      5 & \elbow & \upelb \\
      4 & \elbow & \elbow & \upelb \\
      3 & \cross & \cross & \elbow & \upelb \\
      2 & \elbow & \elbow & \elbow & \elbow & \upelb \\
      1 & \elbow & \cross & \cross & \elbow & \cross & \upelb \\
      & 1 & 2 & 3 & 4 & 5 & 6} &
    \pipes{%
      6 & \upelb \\
      5 & \elbow & \upelb \\
      4 & \elbow & \elbow & \upelb \\
      3 & \cross & \cross & \elbow & \upelb \\
      2 & \cross & \elbow & \elbow & \elbow & \upelb \\
      1 & \elbow & \elbow & \cross & \elbow & \cross & \upelb \\
      & 1 & 2 & 3 & 4 & 5 & 6} &
    \pipes{%
      6 & \upelb \\
      5 & \elbow & \upelb \\
      4 & \elbow & \elbow & \upelb \\
      3 & \cross & \cross & \elbow & \upelb \\
      2 & \elbow & \elbow & \elbow & \cross & \upelb \\
      1 & \elbow & \cross & \cross & \elbow & \elbow & \upelb \\
      & 1 & 2 & 3 & 4 & 5 & 6} &
    \pipes{%
      6 & \upelb \\
      5 & \elbow & \upelb \\
      4 & \elbow & \elbow & \upelb \\
      3 & \cross & \cross & \cross & \upelb \\
      2 & \elbow & \elbow & \elbow & \elbow & \upelb \\
      1 & \elbow & \cross & \cross & \elbow & \elbow & \upelb \\
      & 1 & 2 & 3 & 4 & 5 & 6} &
    \pipes{%
      6 & \upelb \\
      5 & \elbow & \upelb \\
      4 & \elbow & \elbow & \upelb \\
      3 & \cross & \elbow & \elbow & \upelb \\
      2 & \cross & \elbow & \cross & \cross & \upelb \\
      1 & \elbow & \elbow & \cross & \elbow & \elbow & \upelb \\
      & 1 & 2 & 3 & 4 & 5 & 6} \\ \\
    \pipes{%
      6 & \upelb \\
      5 & \elbow & \upelb \\
      4 & \elbow & \elbow & \upelb \\
      3 & \cross & \cross & \elbow & \upelb \\
      2 & \cross & \cross & \elbow & \elbow & \upelb \\
      1 & \elbow & \elbow & \elbow & \elbow & \cross & \upelb \\
      & 1 & 2 & 3 & 4 & 5 & 6} &
    \pipes{%
      6 & \upelb \\
      5 & \elbow & \upelb \\
      4 & \elbow & \elbow & \upelb \\
      3 & \cross & \cross & \elbow & \upelb \\
      2 & \cross & \cross & \elbow & \cross & \upelb \\
      1 & \elbow & \elbow & \elbow & \elbow & \elbow & \upelb \\
      & 1 & 2 & 3 & 4 & 5 & 6} &
    \pipes{%
      6 & \upelb \\
      5 & \elbow & \upelb \\
      4 & \elbow & \elbow & \upelb \\
      3 & \cross & \cross & \cross & \upelb \\
      2 & \cross & \elbow & \elbow & \elbow & \upelb \\
      1 & \elbow & \elbow & \cross & \elbow & \elbow & \upelb \\
      & 1 & 2 & 3 & 4 & 5 & 6} &
    \pipes{%
      6 & \upelb \\
      5 & \elbow & \upelb \\
      4 & \elbow & \elbow & \upelb \\
      3 & \cross & \cross & \cross & \upelb \\
      2 & \cross & \cross & \elbow & \elbow & \upelb \\
      1 & \elbow & \elbow & \elbow & \elbow & \elbow & \upelb \\
      & 1 & 2 & 3 & 4 & 5 & 6} &
    \pipes{%
      6 & \upelb \\
      5 & \elbow & \upelb \\
      4 & \elbow & \elbow & \upelb \\
      3 & \cross & \cross & \elbow & \upelb \\
      2 & \cross & \elbow & \elbow & \cross & \upelb \\
      1 & \elbow & \elbow & \cross & \elbow & \elbow & \upelb \\
      & 1 & 2 & 3 & 4 & 5 & 6}
  \end{array}
\end{displaymath}
\caption{\label{fig:PD} The $15$ elements of $\PD(146235)$.}
\end{figure}

Let $\PD(w)$ denote the set of pipe dreams of shape $w$. When $w \in \mathcal{S}_{\infty}$ fixes $i$ for all $i \geq \ell$, we omit the sea of waves above the antidiagonal connecting $(0,\ell)$ with $(\ell,0)$. For example, the pipe dreams of shape $146235 \in \mathcal{S}_{6}$ are given Figure~\ref{fig:PD}.

To each pipe dream $P$ we associate the weak composition $\wt(P)$ whose $i$th component is equal to the number of crosses in the $i$th row of $P$. For example, the weights of the first column of pipe dreams in Figure~\ref{fig:PD} are $(3,2,0,0,0),(3,0,2,0,0),(1,2,2,0,0)$.

\begin{definition}[\cite{BB93}]
  For $w$ a permutation with no descents at or beyond $n$, the \emph{Schubert polynomial} $\sch_{w} = \sch_{w}(x_1,\ldots,x_n)$ is given by
  \begin{equation}
    \sch_{w} = \sum_{P \in \PD(w)} x^{\wt(P)},
    \label{e:schubert}
  \end{equation}
  where $x^{a}$ is the monomial $x_1^{a_1} \cdots x_n^{a_n}$.
  \label{def:schubert}
\end{definition}

For example, from Figure~\ref{fig:PD} we can compute
\begin{eqnarray}
  \sch_{(146235)} & = &  x_1^3 x_2^2  + x_1^3 x_3^2 +  x_1^2 x_2^3  +   2 x_1^2 x_2^2 x_3  +   2 x_1^2 x_2 x_3^2  +   x_1^2 x_3^3 \\\label{e:sch32_x}
  & & +  x_1^3 x_2 x_3 + x_1 x_2^3 x_3  +   2 x_1 x_2^2 x_3^2  +   x_1 x_2 x_3^3  +   x_2^3 x_3^2  +   x_2^2 x_3^3. \nonumber
\end{eqnarray}

Comparing \eqref{e:sch32_x} with \eqref{e:s32_x}, we see that $\sch_{(146235)} = s_{(3,2)}(x_1,x_2,x_3)$. Indeed, this is not a coincidence. For $\lambda$ a partition of length at most $n$, let $v(\lambda,n)$ be the permutation with a unique descent at position $n$ and values $i + \lambda_{n+1-i}$ at $1 \leq i \leq n$. For example, $v((3,2),3) = 146235$. This map is invertible on permutations with at most one descent, and we call such permutations \emph{Grassmannian}. 

\begin{theorem}[\cite{LS82}]
  For $\lambda$ a partition of length at most $n$, we have
  \begin{equation}
    \sch_{v(\lambda,n)} = s_{\lambda}(x_1,\ldots,x_n).
  \end{equation}
\end{theorem}

That is to say, Schur polynomials represent the Schubert classes of the Grassmannian manifold. In light of this, one may regard Schubert polynomials as a lifting of Schur polynomials from $\Lambda_n$ to $\Poly_n$. Since Schur polynomials are well understood in comparison to Schubert polynomials, our aim is to lift tools and techniques from symmetric polynomials in order to gain better insights into Schubert polynomials. Of course, since Schubert polynomials are not symmetric, the challenge lies in choosing what to lift and how to lift it.

\subsection{Slide polynomials}
\label{sec:schub-slide}

To define our new bases for polynomials that lift quasisymmetric polynomials, we begin with a few operations on weak compositions. For $a$ a weak composition, let $\flatten(a)$, called the \emph{flattening} of $a$, be the (strong) composition obtained by removing all $0$ terms. Given weak compositions $a,b$ of length $n$, we say that $b$ \emph{dominates} $a$, denoted by $b \geq a$, if
\begin{equation}
  b_1 + \cdots + b_i \geq a_1 + \cdots + a_i
\end{equation}
for all $i=1,\ldots,n$. Note that this extends the usual dominance order on partitions.

\begin{definition}
  For a weak composition $a$ of length $n$, define the \emph{monomial slide polynomial} $\Mono_{a} = \Mono_{a}(x_1,\ldots,x_n)$ by
  \begin{equation}
    \Mono_{a} = \sum_{\substack{b \geq a \\ \flatten(b) = \flatten(a)}} x_1^{b_1} \cdots x_n^{b_n},
    \label{e:monomial-shift}
  \end{equation}
  where the sum is over all compositions $b$ obtained by shifting the entries of $a$ to the left while preserving their relative order.
  \label{def:monomial}
\end{definition}

For example, we have
\begin{equation}
  \Mono_{(0,2,0,3)} = x_1^2 x_2^3 + x_1^2 x_3^3 + x_1^2 x_4^3 + x_2^2 x_3^3 + x_2^2 x_4^3.
\end{equation}
Note that this polynomial is \emph{not} quasisymmetric; it is missing the term $x_3^2 x_4^3$. 

We say that a weak composition $a$ is \emph{quasi-flat} if the nonzero terms occur in an interval. For example, $(0,2,0,3)$ is not quasi-flat, but $(0,0,2,3)$ is.

\begin{lemma}
  For a weak composition $a$ of length $n$, let $k$ be the index of the last nonzero term of $a$. Then $\Mono_a$ is quasisymmetric in $x_1,\ldots,x_k$ if and only if $a$ is quasi-flat. Moreover, in this case, we have $\Mono_a = M_{\flatten(a)}(x_1,\ldots,x_k)$.
  \label{lem:lift-mono}
\end{lemma}

\begin{proof}
Suppose $a$ is not quasi-flat, i.e., $a_{i-1}> a_i = 0$ for some $i \leq k$. Then the term $x_1^{a_1} \ldots x_{i-2}^{a_{i-2}}x_{i-1}^{a_{i-1}}x_{i+1}^{a_{i+1}}\ldots x_k^{a_k}$ appears in $\Mono_a$ but the term $x_1^{a_1} \ldots x_{i-2}^{a_{i-2}}x_i^{a_{i-1}}x_{i+1}^{a_{i+1}}\ldots x_k^{a_k}$ does not, hence $\Mono_a$ is not quasisymmetric in $x_1,\ldots,x_{k}$.
  
  Conversely, suppose $a$ is quasi-flat. Then every $b$ with last nonzero entry at or before $k$ for which $\flatten(b)=\flatten(a)$ dominates $a$, so $\Mono_a = M_{\flatten(a)}(x_1,\ldots,x_k)$. 
\end{proof}

For example, $\Mono_{(0,0,2,3)} = x_1^2 x_2^3 + x_1^2 x_3^3 + x_1^2 x_4^3 + x_2^2 x_3^3 + x_2^2 x_4^3 + x_3^2 x_4^2 = M_{(2,3)}$. 

Not only do the monomial slide polynomials lift the monomial quasisymmetric polynomials from $\QSym_n^k$ to $\Poly_n^k$, but they form a $\mathbb{Z}$-basis for $\Poly_n^k$ as well.

\begin{theorem}
  The monomial slide polynomials $\{ \Mono_{a} \mid  a=(a_1,\ldots,a_n) \mbox{ and } \sum a_i=k \}$ form a $\mathbb{Z}$-basis for $\Poly_n^k$.
  \label{thm:basis-mono}
\end{theorem}

\begin{proof}
  Using dominance order on compositions, there exist nonnegative integers $c_{a,b}$ such that
  \begin{displaymath}
    \Mono_a = x^a + \sum_{b>a} c_{a,b} x^b.
  \end{displaymath}
  In particular, since dominance is a suborder of reverse lexicographic order, the monomial slide polynomials $\{\Mono_a\}$ are upper uni-triangular with respect to the monomials $\{x^a\}$. Since the latter clearly form a $\mathbb{Z}$-basis for $\Poly_n^k$, so do the former.
\end{proof}

We now lift the fundamental quasisymmetric basis in a similar manner.

\begin{definition}
  For a weak composition $a$ of length $n$, define the \emph{fundamental slide polynomial} $\Fund_{a} = \Fund_{a}(x_1,\ldots,x_n)$ by
  \begin{equation}
    \Fund_{a} = \sum_{\substack{b \geq a \\ \flatten(b) \ \mathrm{refines} \ \flatten(a)}} x_1^{b_1} \cdots x_n^{b_n},
    \label{e:fundamental-shift}
  \end{equation}
  where the sum is over all compositions $b$ obtained by shifting or splitting the entries of $a$ to the left while preserving their relative order.
  \label{def:fundamental-shift}
\end{definition}

For example, we have
\begin{eqnarray}
  \Fund_{(0,2,0,3)} & = & x_1^2 x_2^3 + x_1^2 x_3^3 + x_1^2 x_4^3 + x_2^2 x_3^3 + x_2^2 x_4^3 + x_1^2 x_2 x_3^2 + x_1^2 x_2 x_4^2  \label{e:F_x}\\
  & & + x_1^2 x_3 x_4^2 + x_2^2 x_3 x_4^2 + x_1^2 x_2^2 x_3 + x_1^2 x_2^2 x_4 + x_1^2 x_3^2 x_4 + x_2^2 x_3^2 x_4 \nonumber\\
  & & + x_1 x_2 x_3^3 + x_1 x_2 x_4^3 + x_1 x_2 x_3 x_4^2 + x_1 x_2 x_3^2 x_4 + x_1^2 x_2 x_3 x_4. \nonumber
\end{eqnarray}

As with their quasisymmetric counter-parts, it is more convenient to expand the fundamental slide polynomials in terms of the monomial slide basis. To do this, we require a further refinement of dominance. Given weak compositions $a,b$ of length $n$, we say that $b$ \emph{strongly dominates} $a$, denoted by $b \sgeq a$, if $b \geq a$ and for all $c \geq a$ such that $\flatten(c)=\flatten(b)$, we have $c \geq b$ as well. This definition makes the following statement true.

\begin{proposition}
  For $a$ a weak composition of length $n$, we have
  \begin{equation}
    \Fund_{a} = \sum_{\substack{b \sgeq a \\ \flatten(b) \ \mathrm{refines} \ \flatten(a)}} \Mono_{b}.
  \end{equation}
  \label{prop:F-to-M}
\end{proposition}

For example, \eqref{e:F_x} can be written more compactly as
\begin{eqnarray}
  \Fund_{(0,2,0,3)} & = & \Mono_{(0,2,0,3)} + \Mono_{(0,2,1,2)} + \Mono_{(0,2,2,1)} + \Mono_{(1,1,0,3)} \label{e:F_M} \\
  & & + \Mono_{(1,1,1,2)} + \Mono_{(1,1,2,1)} + \Mono_{(2,1,1,1)}. \nonumber  
\end{eqnarray}

As with the monomial slide basis, we have the following characterization of when a fundamental slide polynomial is quasisymmetric.

\begin{lemma}
  For a weak composition $a$, let $k$ be the index of the last nonzero term of $a$. Then $\Fund_a$ is quasisymmetric in $x_1,\ldots,x_k$ if and only if $a$ is quasi-flat. Moreover, in this case, we have $\Fund_a = F_{\flatten(a)}(x_1,\ldots,x_k)$.
  \label{lem:lift-fund}
\end{lemma}

\begin{proof}
  If $a$ is not quasi-flat, then the same term highlighted in the proof of Lemma~\ref{lem:lift-mono} that is missing from $\Mono_a$ to make it quasisymmetric is missing from $\Fund_a$ as well, and hence it is not quasisymmetric in $x_1,\ldots,x_k$.

  As in the monomial case, if $a$ is quasi-flat, then any $b$ with last nonzero entry at or before $k$ for which $\flatten(b)$ refines $\flatten(a)$ necessarily dominates $a$, and the minimal such element is also quasi-flat. Combining Lemma~\ref{lem:lift-mono} and \eqref{e:fundamental}, we have $\Fund_a = F_{\flatten(a)}(x_1,\ldots,x_k)$.
\end{proof}

\begin{theorem}
  The fundamental slide polynomials $\{ \Fund_{a} \mid  a=(a_1,\ldots,a_n) \mbox{ and } \sum a_i=k \}$ form a $\mathbb{Z}$-basis for $\Poly_n^k$.
  \label{thm:basis-fund}
\end{theorem}

\begin{proof}
  Using dominance order on compositions, there exist nonnegative integers $c_{a,b}$ such that
  \begin{displaymath}
    \Fund_a = \Mono_a + \sum_{b>a} c_{a,b} \Mono_b.
  \end{displaymath}
  Since dominance is a suborder of reverse lexicographic order, the fundamental slide polynomials $\{\Fund_a\}$ are upper uni-triangular with respect to the monomial slide polynomials $\{\Mono_a\}$. By Theorem~\ref{thm:basis-mono}, the latter form a $\mathbb{Z}$-basis for $\Poly_n^k$, hence so do the former.
\end{proof}

\subsection{Quasi-Yamanouchi pipe dreams}
\label{sec:schub-quasi}

The expansion in \eqref{e:schubert}, while beautifully combinatorial, is limited in the same ways as \eqref{e:schur}. In particular, it makes calculations somewhat intractable. Parallel to Gessel's expansion for the Schur polynomial $s_{\lambda}$ in terms of fundamental quasisymmetric polynomials $F_{\alpha}$, we now give the expansion for the Schubert polynomial $\sch_{w}$ in terms of the fundamental slide basis $\Fund_{a}$. We begin by generalizing the quasi-Yamanouchi condition on semistandard Young tableaux to a condition on pipe dreams.

\begin{definition}
  A pipe dream is \emph{quasi-Yamanouchi} if, for every $i$, the westernmost $\cross$ in row $i$ is in the first column or lies weakly west of some $\cross$ in the $i+1$st row. Let $\QPD(w)$ denote the set of quasi-Yamanouchi pipe dreams of shape $w$.
  \label{def:quasi-Yam-pipe}
\end{definition}

For example, looking back at Figure~\ref{fig:PD}, there are five quasi-Yamanouchi pipe dreams of shape $146235$ as shown in Figure~\ref{fig:QPD}.

\begin{figure}[ht]
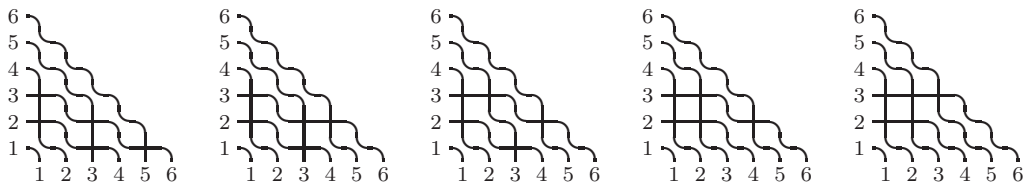

\begin{displaymath}
  \begin{array}{ccccc}
    \pipes{%
      6 & \upelb  \\
      5 & \elbow & \upelb \\
      4 & \elbow & \elbow & \upelb \\
      3 & \cross & \elbow & \elbow & \upelb \\
      2 & \cross & \elbow & \cross & \elbow & \upelb \\
      1 & \elbow & \elbow & \cross & \elbow & \cross & \upelb \\
      & 1 & 2 & 3 & 4 & 5 & 6} &
    \pipes{%
      6 & \upelb \\
      5 & \elbow & \upelb \\
      4 & \elbow & \elbow & \upelb \\
      3 & \cross & \elbow & \elbow & \upelb \\
      2 & \cross & \elbow & \cross & \cross & \upelb \\
      1 & \elbow & \elbow & \cross & \elbow & \elbow & \upelb \\
      & 1 & 2 & 3 & 4 & 5 & 6} &
    \pipes{%
      6 & \upelb \\
      5 & \elbow & \upelb \\
      4 & \elbow & \elbow & \upelb \\
      3 & \cross & \cross & \elbow & \upelb \\
      2 & \cross & \elbow & \elbow & \cross & \upelb \\
      1 & \elbow & \elbow & \cross & \elbow & \elbow & \upelb \\
      & 1 & 2 & 3 & 4 & 5 & 6} &
    \pipes{%
      6 & \upelb \\
      5 & \elbow & \upelb \\
      4 & \elbow & \elbow & \upelb \\
      3 & \cross & \cross & \elbow & \upelb \\
      2 & \cross & \cross & \elbow & \cross & \upelb \\
      1 & \elbow & \elbow & \elbow & \elbow & \elbow & \upelb \\
      & 1 & 2 & 3 & 4 & 5 & 6} &
    \pipes{%
      6 & \upelb \\
      5 & \elbow & \upelb \\
      4 & \elbow & \elbow & \upelb \\
      3 & \cross & \cross & \cross & \upelb \\
      2 & \cross & \cross & \elbow & \elbow & \upelb \\
      1 & \elbow & \elbow & \elbow & \elbow & \elbow & \upelb \\
      & 1 & 2 & 3 & 4 & 5 & 6} 
  \end{array}
\end{displaymath}
\caption{\label{fig:QPD} The $5$ elements of $\QPD(146235)$.}
\end{figure}

Analogous to the case for tableaux, we define a surjective \emph{destandardization} map from pipe dreams to quasi-Yamanouchi pipe dreams.

\begin{definition}
  For $P\in\PD(w)$, the \emph{destandardization of $P$}, denoted by $\destand(P)$, is the pipe dream constructed from $P$ as follows. For each row, say $i-1$, with no $\cross$ in the first column, if every $\cross$ in row $i-1$ lies strictly east of every $\cross$ in row $i$, then move every $\cross$ in row $i-1$ northwest one position. Repeat until no such row exists.
  \label{def:w-destand}  
\end{definition}

\begin{lemma}
  The destandardization map is well-defined and satisfies the following:
  \begin{enumerate}
  \item for $P \in \PD(w)$, $\destand(P) \in \QPD(w)$;
  \item for $P \in \PD(w)$, $\destand(P)=P$ if and only if $P \in \QPD(w)$;
  \item $\destand:\PD(w) \rightarrow \QPD(w)$ is surjective;
  \item $\destand:\PD(w) \rightarrow \QPD(w)$ is injective if and only if $w_i<w_{i+1}$ for all $i\ge w^{-1}(1)$.
  \end{enumerate}
  \label{lem:w-destand}
\end{lemma}

\begin{proof}
  The reduced condition on pipe dreams, that no two pipes cross more than once, ensure that when every $\cross$ in row $i-1$ lies strictly east of every $\cross$ in row $i$, there is no $\cross$ immediately northwest of the westernmost $\cross$ of row $i-1$. Therefore the map is indeed well-defined.
  
  The process of destandardization terminates only if the quasi-Yamanouchi condition is satisfied, proving (1) and (2), and property (3) follows from (2). 

  For property (4), note that for any $w$, there is a pipe dream $P_{L(w)}$ given by placing $L(w)_i$ $\cross$'s flush left in row $i$. Suppose $w$ has no descent after $w^{-1}(1)=m$. Then $P_{L(w)}$ has $\cross$'s in row $i$, column $1$ for all $i<m$, and no $\cross$'s in row $i$ for $i\ge m$. It is then immediate from the description of the local moves connecting $\PD(w)$ (\cite{BB93}) that all pipe dreams for $w$ have $\cross$'s in row $i$, column $1$ for all $i<m$, and no $\cross$'s in row $i$ for $i\ge m$. Thus $\destand(P) = P$ for all $P\in\PD(w)$.   
Conversely, suppose $w$ has a descent after $w^{-1}(1)=m$, and let $i$ be the position of the earliest such descent. Then $P_{L(w)}$ has a $\cross$ in row $i$ but no $\cross$'s in row $i-1$. Another pipe dream for $w$ can then be obtained from $P_{L(w)}$ by shifting all $\cross$'s in row $i$ southeast one position. This pipe dream is not quasi-Yamanouchi.
\end{proof}

\begin{theorem}
  For $w$ any permutation, we have
  \begin{equation}
    \sch_{w} = \sum_{P \in \QPD(w)} \Fund_{\wt(P)}.
    \label{e:schubert-slide}
  \end{equation}
  \label{thm:schubert-slide}
\end{theorem}

\begin{proof}
  Note that if $\destand(P)=Q$, then $\wt(P) \geq \wt(Q)$ and $\flatten(\wt(P))$ refines $\flatten(\wt(Q))$ since $Q$ is obtained by moving \emph{all} $\cross$'s in row $i-1$ to row $i$. Conversely, we claim that given $Q \in \QPD(w)$, for every weak composition $b$ of length $n$ such that $b \geq \wt(Q)$ and $\flatten(b)$ refines $\flatten(\wt(Q))$, there is a unique $P \in \PD(w)$ with $\wt(P) = b$ such that $\destand(P) = Q$. From the claim, for $Q \in \QPD(w)$, we have
  \begin{displaymath}
    \sum_{P \in \destand^{-1}(Q)} x^{\wt(P)} = \Fund_{\wt(Q)}.
  \end{displaymath}
  The theorem follows from this and Lemma~\ref{lem:w-destand}.
  
  To construct $P$ from $b$ and $Q$, for $j = 1,\ldots,n$, if $\wt(Q)_{j} = b_{i_{j-1} + 1} + \cdots + b_{i_{j}}$, then, from east to west, slide the first $b_{i_{j-1} + 1}$ $\cross$'s southeast to row $i_{j-1} + 1$, the next $b_{i_{j-1} + 2}$ $\cross$'s southeast to row $i_{j-1} + 2$, and so on. Existence is proved, and uniqueness follows from the lack of choice at every step.
\end{proof}

For example, from Figure~\ref{fig:QPD} we can compute
\begin{equation}
  \sch_{(146235)} = \Fund_{(2,2,1,0,0)} + \Fund_{(1,3,1,0,0)} + \Fund_{(1,2,2,0,0)} + \Fund_{(0,3,2,0,0)} + \Fund_{(0,2,3,0,0)}.
  \label{e:sch32_F}
\end{equation}

Of course, since $146235$ is a Grassmannian permutation, this is the same example as the running example of $\lambda=(3,2)$ in Section~\ref{sec:poly}.

For a non-Grassmannian example, the Schubert polynomial for $w = 135264$ is
\begin{eqnarray}
  \sch_{(135264)} & = & x_1^2 x_2^2+2 x_1^2 x_2 x_3+x_1^2 x_2 x_4+x_1^2 x_2 x_5+x_1^2 x_3^2+x_1^2 x_3 x_4\\\nonumber
  & & +x_1^2 x_3 x_5+2 x_1 x_2^2 x_3 +x_1 x_2^2 x_4+x_1 x_2^2 x_5+2 x_1 x_2 x_3^2\\\nonumber
  & & +2 x_1 x_2 x_3 x_4+2 x_1 x_2 x_3 x_5+x_1 x_3^2 x_4+x_1 x_3^2 x_5 +x_2^2 x_3^2\\\nonumber
  & & +x_2^2 x_3 x_4+x_2^2 x_3 x_5+x_2 x_3^2 x_4+x_2 x_3^2 x_5.
\end{eqnarray}
The $25$ terms in the monomial expansion correspond to the $25$ pipe dreams for $w$, of which only the $5$ shown in Figure~\ref{fig:QPD-graph} are quasi-Yamanouchi. Thus we have the following compacted expansion in terms of fundamental slide polynomials,
\begin{equation}
  \sch_{(135264)} = \Fund_{(1,1,2,0,0)} + \Fund_{(1,2,1,0,0)} + \Fund_{(0,2,2,0,0)} + \Fund_{(0,2,1,0,1)} + \Fund_{(0,1,2,0,1)}.
  \label{e:sch_F}
\end{equation}

\begin{figure}[ht]
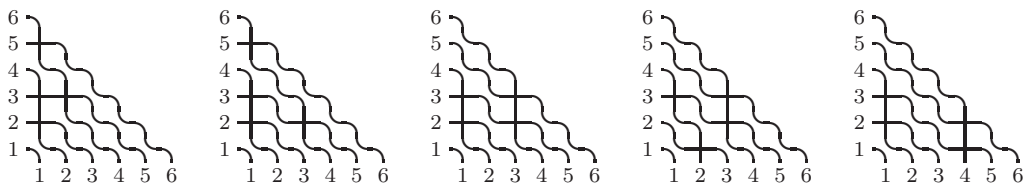

  \begin{displaymath}
    \begin{array}{ccccc}
      \pipes{%
        6 & \upelb \\
        5 & \cross & \upelb \\
        4 & \elbow & \elbow & \upelb \\
        3 & \cross & \cross & \elbow & \upelb \\
        2 & \cross & \elbow & \elbow & \elbow & \upelb \\
        1 & \elbow & \elbow & \elbow & \elbow & \elbow & \upelb \\
        & 1 & 2 & 3 & 4 & 5 & 6} &
      \pipes{%
        6 & \upelb \\
        5 & \cross & \upelb \\
        4 & \elbow & \elbow & \upelb \\
        3 & \cross & \elbow & \elbow & \upelb \\
        2 & \cross & \elbow & \cross & \elbow & \upelb \\
        1 & \elbow & \elbow & \elbow & \elbow & \elbow & \upelb \\
        & 1 & 2 & 3 & 4 & 5 & 6} &
      \pipes{%
        6 & \upelb \\
        5 & \elbow & \upelb \\
        4 & \elbow & \elbow & \upelb \\
        3 & \cross & \elbow & \cross & \upelb \\
        2 & \cross & \elbow & \cross & \elbow & \upelb \\
        1 & \elbow & \elbow & \elbow & \elbow & \elbow & \upelb \\
        & 1 & 2 & 3 & 4 & 5 & 6} &
      \pipes{%
        6 & \upelb \\
        5 & \elbow & \upelb \\
        4 & \elbow & \elbow & \upelb \\
        3 & \cross & \elbow & \cross & \upelb \\
        2 & \elbow & \elbow & \cross & \elbow & \upelb \\
        1 & \elbow & \cross & \elbow & \elbow & \elbow & \upelb \\
        & 1 & 2 & 3 & 4 & 5 & 6} &
      \pipes{%
        6 & \upelb \\
        5 & \elbow & \upelb \\
        4 & \elbow & \elbow & \upelb \\
        3 & \cross & \elbow & \elbow & \upelb \\
        2 & \cross & \elbow & \elbow & \cross & \upelb \\
        1 & \elbow & \elbow & \elbow & \cross & \elbow & \upelb \\
        & 1 & 2 & 3 & 4 & 5 & 6}
    \end{array}
\end{displaymath}
\caption{\label{fig:QPD-graph} The quasi-Yamanouchi pipe dreams for $w=135264$.}
\end{figure}

The fundamental slide basis also has a triangularity with respect to the Schubert basis that makes changing between the bases computationally efficient.

\begin{proposition}
  For $w$ any permutation, there exist nonnegative integer coefficients $c_{w,b}$ such that
  \begin{equation}
    \sch_{w} = \Fund_{L(w)} + \sum_{b > L(w)} c_{w,b} \Fund_{b},
    \label{e:lead}
  \end{equation}
  where $L(w)$ is the Lehmer code of $w$.
  \label{prop:lead}
\end{proposition}

\begin{proof}
  The leading monomial for $\sch_w$ in reverse lexicographic order is $x^{L(w)}$ \cite{Mac91}. The result follows from the triangularity of fundamental slide polynomials with respect to monomials mentioned in the proof of Theorem~\ref{thm:basis-fund}.
\end{proof}

%
\section{Stability}
%
\label{sec:stable}

\subsection{Increasing the number of variables}
\label{sec:stable-vars}

For $w$ a permutation of $\mathcal{S}_n$ and $m$ a nonnegative integer, let $w \times 1^m$ denote the permutation of $\mathcal{S}_{n+m}$ given by $w_1 \cdots w_n (n+1) \cdots (n+m)$. Lascoux and Sch\"{u}tzenberger \cite{LS82} used the following stability property of Schubert polynomials to show that Schubert polynomials defined by divided difference operators are well-defined (though, from the pipe dream perspective the result is far easier to see). 

\begin{theorem}[\cite{LS82}]
  For $w$ a permutation of $\mathcal{S}_n$ and $m$ a nonnegative integer, we have
  \begin{equation}
    \sch_{w \times 1^m} = \sch_w.
  \end{equation}
  \label{thm:schub-stable}
\end{theorem}

Note that adding variables in the general polynomial setting is not the same as in the symmetric polynomial setting. The analog for Schur polynomials is the stability
\begin{displaymath}
  s_{\lambda}(x_1,\ldots,x_n,0,\ldots,0) = s_{\lambda}(x_1,\ldots,x_n).
\end{displaymath}
The analogous property for quasisymmetric functions is the following stability
\begin{eqnarray*}
  M_{\alpha}(x_1,\ldots,x_n,0,\ldots,0) & = & M_{\alpha}(x_1,\ldots,x_n),\\
  F_{\alpha}(x_1,\ldots,x_n,0,\ldots,0) & = & F_{\alpha}(x_1,\ldots,x_n).
\end{eqnarray*}

The slide polynomials exhibit a parallel stability property to that of Schubert polynomials. For a weak composition $a$ and a nonnegative integer $m$, let $a \times 0^m = (a_1,\ldots,a_n,0,\ldots,0)$ be the composition of length $n+m$ obtained by appending $m$ zeros to the end of $a$.

\begin{theorem}
  Let $a$ be a weak composition and $m$ a nonnegative integer. Then we have
  \begin{equation}
    \Mono_{a \times 0^m} = \Mono_{a} \hspace{1em} \mbox{and} \hspace{1em} \Fund_{a \times 0^m} = \Fund_{a}.
  \end{equation}
  \label{thm:slide-stable}
\end{theorem}

\begin{proof}
  For $a$ a weak composition of length $n$ and $b$ a weak composition of length $n+m$, $b \geq a \times 0^m$ if and only if $(b_1,\ldots,b_n) \geq a$ and $b_i=0$ for all $i>n$. The result now follows from the definitions of $\Mono_a$ and $\Fund_a$.
\end{proof}

\subsection{Stable Limits}
\label{sec:stable-limit}

In this section we consider a different stability, the one that gives rise to symmetric \emph{functions}, i.e.
\begin{equation}
  \lim_{n \rightarrow \infty} s_{\lambda}(x_1,\ldots,x_n) = s_{\lambda}(X).
  \label{e:stable-s}
\end{equation}

To lift this limit to Schubert polynomials, begin by noticing that $v(\lambda,n+m) = 1^m \times v(\lambda,n)$, where for $w$ a permutation of $\mathcal{S}_n$ and $m$ a nonnegative integer, $1^m \times w$ denotes the permutation of $\mathcal{S}_{n+m}$ given by $1 \cdots m (w_1+m) \cdots (w_n+m)$. In general, we wish to consider the limit (if it exists) of the Schubert polynomial $\sch_{1^m \times w}$ as $m$ grows. For $w$ a Grassmannian permutation, we may re-write \eqref{e:stable-s} as
\begin{equation}
  \lim_{m \rightarrow \infty} \sch_{1^m \times v(\lambda,n)} = \lim_{m \rightarrow \infty} \sch_{v(\lambda,n+m)} = \lim_{n \rightarrow \infty} s_{\lambda}(x_1,\ldots,x_n) = s_{\lambda}(X).
  \label{e:stable-sch}
\end{equation}

For the general case, recall the set of \emph{reduced decompositions for $w$}, denoted by $\RD(w)$, is the set of $\ell$-tuples $(s_{i_1},\ldots,s_{i_{\ell}})$ for which $w = s_{i_{\ell}} \cdots s_{i_1}$, where $s_i$ is the simple transposition swapping $i$ and $i+1$ and $\ell = \ell(w)$ is the number of inversions of $w$.

For example, the reduced decompositions for $w = 24153$ are
\begin{equation}
  \RD(w) = \{s_1 s_3 s_2 s_4, s_1 s_3 s_4 s_2, s_3 s_1 s_4 s_2, s_3 s_1 s_2 s_4, s_3 s_4 s_1 s_2\}.
\end{equation}

Stanley \cite{Sta84} defined symmetric functions indexed by permutations. To avoid confusion with fundamental quasisymmetric functions, we diverge from standard notation of $F_w$ and denote the Stanley symmetric functions by $\sta_{w}$. Also note that we follow usual conventions and have our $\sta_w = F_{w^{-1}}$ in \cite{Sta84}.

\begin{definition}[\cite{Sta84}]
  For $w$ a permutation, the \emph{Stanley symmetric function of $w$}, denoted by $\sta_{w}$, is
  \begin{equation}
    \sta_{w}(X) = \sum_{\sigma \in \RD(w)} F_{\Des(\sigma)}(X),
  \end{equation}
  where $\Des(\sigma)$ is the descent composition of the reversed sequence of indices of $\sigma$, i.e. $\Des(s_{i_{\ell}}\cdots s_{i_{1}}) = \Des(i_1,\ldots,i_{\ell})$.
\end{definition}

For example, we compute the Stanley symmetric function for $w = 24153$ to be
\begin{equation}
  \sta_{24153}(X) = F_{(2,2)}(X) + F_{(2,1,1)}(X) + 2 F_{(1,2,1)}(X) + F_{(1,1,2)}(X).
\end{equation}

Not only are the Stanley symmetric functions honest symmetric functions \cite{Sta84}, Edelman and Greene \cite{EG87} showed that they are, in fact, Schur positive. For example, $\sta_{24153}(X) = s_{(2,2)}(X) + s_{(2,1,1)}(X)$. Furthermore, Macdonald \cite{Mac91} showed that Stanley symmetric functions are the \emph{stable limits} of Schubert polynomials. 

\begin{theorem}[\cite{Mac91}]
  For $w$ a permutation of $\mathcal{S}_n$, we have
  \begin{equation}
    \lim_{m \rightarrow \infty} \sch_{1^m \times w} = \sta_{w}(X).
  \end{equation}
  \label{thm:schub-limit}
\end{theorem}

The monomial and fundamental quasisymmetric polynomials exhibit a parallel stability to Schur polynomials, namely, 
\begin{eqnarray}
  \lim_{n \rightarrow \infty} M_{\alpha}(x_1,\ldots,x_n) & = & M_{\alpha}(X), \label{e:stable-M}\\
  \lim_{n \rightarrow \infty} F_{\alpha}(x_1,\ldots,x_n) & = & F_{\alpha}(X). \label{e:stable-F}
\end{eqnarray}

The slide polynomials exhibit an analogous stability. It is easy to see that $L(1^m \times w) = 0^m \times L(w)$, so we let $0^m \times a = (0,\ldots,0,a_1,\ldots,a_n)$ be the composition of length $n+m$ obtained by prepending $m$ zeros to $a$. Then we have the following stability result for slide polynomials.

\begin{theorem}
  For a weak composition $a$, we have
  \begin{eqnarray}
    \lim_{m\rightarrow\infty}\Mono_{0^m \times a} & = & M_{\flatten(a)}(X),\\
    \lim_{m\rightarrow\infty}\Fund_{0^m \times a} & = & F_{\flatten(a)}(X).
  \end{eqnarray}
  \label{thm:stable-limit}
\end{theorem}

\begin{proof}
  Let $\ell = \ell(\flatten(a))$ be the number of nonzero terms of $a$. Then for all $m>0$, by Lemmas~\ref{lem:lift-mono} and \ref{lem:lift-fund}, we have
  \begin{eqnarray*}
    \Mono_{0^m \times a}(x_1,\ldots,x_{m+\ell},0,\ldots,0) & = & M_{\flatten(a)}(x_1,\ldots,x_{m+\ell},0,\ldots,0) \\ & = & M_{\flatten(a)}(x_1,\ldots,x_{m+\ell}),\\
    \Fund_{0^m \times a}(x_1,\ldots,x_{m+\ell},0,\ldots,0) & = & F_{\flatten(a)}(x_1,\ldots,x_{m+\ell},0,\ldots,0) \\ & = & F_{\flatten(a)}(x_1,\ldots,x_{m+\ell}),
  \end{eqnarray*}
  where the latter equalities hold by stability of quasisymmetric polynomials.
\end{proof}

\subsection{A refinement of stability}
\label{sec:stable-refine}

The fundamental slide polynomials provide a useful tool to give an easy proof of Theorem~\ref{thm:schub-limit} by means of a more subtle understanding of the stability. We define a \emph{standardization} map from pipe dreams to reduced decompositions that is injective on the set of quasi-Yamanouchi pipe dreams.

\begin{definition}
  For $P \in \PD(w)$, the \emph{standardization of P}, denoted by $\stand(P)$, is the decomposition obtained by reading the $\cross$'s of $P$ from left to right, top to bottom, and recording the cross in position $(i,j)$ as $s_{i+j-1}$. 
  \label{def:w-stand}  
\end{definition}

For examples of the standardization map, see Figure~\ref{fig:w-stand}. Note that when $w = v(\lambda,n)$ is a Grassmannian permutation, and $\phi:\PD(w) \stackrel{\sim}{\rightarrow} \SSYT_n(\lambda)$ is the usual bijection, we have $\phi(\stand(P)) = \stand(\phi(P))$, where $\stand$ on semi-standard Young tableaux is the usual standardization map on semistandard Young tableaux that gives a standard Young tableaux.

\begin{figure}[ht]
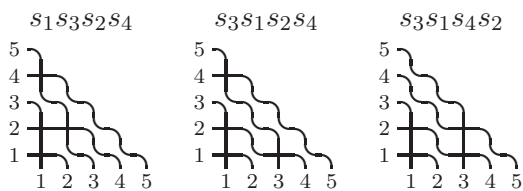

  \begin{displaymath}
    \begin{array}{ccc}
    s_1 s_3 s_2 s_4 &
    s_3 s_1 s_2 s_4 &
    s_3 s_1 s_4 s_2 \\
    \pipes{%
      5 & \upelb \\
      4 & \cross & \upelb \\
      3 & \elbow & \elbow & \upelb \\
      2 & \cross & \cross & \elbow & \upelb \\
      1 & \cross & \elbow & \elbow & \elbow & \upelb \\
      & 1 & 2 & 3 & 4 & 5} &
    \pipes{%
      5 & \upelb \\
      4 & \cross & \upelb \\
      3 & \elbow & \elbow & \upelb \\
      2 & \cross & \elbow & \elbow & \upelb \\
      1 & \cross & \elbow & \cross & \elbow & \upelb \\
      & 1 & 2 & 3 & 4 & 5} &
    \pipes{%
      5 & \upelb \\
      4 & \elbow & \upelb \\
      3 & \elbow & \elbow & \upelb \\
      2 & \cross & \elbow & \cross & \upelb \\
      1 & \cross & \elbow & \cross & \elbow & \upelb \\
      & 1 & 2 & 3 & 4 & 5} 
    \end{array}
  \end{displaymath}
\caption{\label{fig:w-stand} The $3$ elements of $\QPD(24153)$ and their standardizations.}
\end{figure}

While the standardization map is neither injective nor surjective, it splits through quasi-Yamanouchi pipe dreams analogous to the case with tableaux. To make our result clear, we define a left inverse for standardization making use of \emph{virtual pipe dreams}, which are those allowed to have $\cross$'s below the $x$-axis. (We index rows below the $x$-axis by $0$, $-1$, $-2$, etc.)

\begin{definition}
  For $\sigma = s_{i_k} \cdots s_{i_1} \in \RD(w)$, let $\sit(\sigma)$ be the (virtual) pipe dream constructed as follows. Place a $\cross$ in the first column of row $i_1$. Assuming $\cross$'s have been placed for $i_1,\ldots,i_{j-1}$, if $i_j > i_{j-1}$, then place a $\cross$ in the same row and east of the most recently placed $\cross$ so that the row and column indices sum to $j+1$, and if $i_j<i_{j-1}$, then place a $\cross$ in the northernmost row south of the row of the most recently placed cross for which there exists a column such that the row index and column index sum to $j+1$.
  \label{def:w-sit}  
\end{definition}

Note that $\sit(\sigma)$ might indeed give a virtual pipe dream since one might run out of rows before the algorithm terminates and be forced to place a $\cross$ below the $x$-axis.

To help analyze when this happens, we define the following statistic on permutations,
\begin{equation}
  \eta(w) =  \inv(w) - \max(L(w)) +  \delta(w) - \min\{i \mid w_{i} > w_{i+1}\} 
  \label{e:eta}
\end{equation}
where $\delta(w)=0$ if $\max(L(w))$ is attained at some position later than the first descent, and $\delta(w)=1$ otherwise. For example, $\eta(354162) = 8 - 3 + 1 - 2 = 4$. Note that $\eta(1^m \times w) = \eta(w)-m$. For example, $\eta(12576384) = 2$.

\begin{lemma}
  For any permutation $w$, there is a (virtual) quasi-Yamanouchi pipe dream for $w$ with a cross $\eta(w)$ rows below the $x$-axis, and no quasi-Yamanouchi pipe dream for $w$ has a cross in any row lower than this.
  \label{lem:lowest}
\end{lemma}

\begin{proof}
  For $P \in \QPD(w)$, the number of rows of $P$ with at least one $\cross$ is equal to one plus the number of descents of $\stand(P)$, read right to left. Let $\mathrm{des(P)}$ denote the number of descents of $\stand(P)$. Letting $\mathrm{pass}(P)$ denote the number of decreasing runs of $\stand(P)$, read right to left, we have $\mathrm{des}(P) + \mathrm{pass}(P) = \inv(w)$. In particular, $\mathrm{des}(P)$, and so, too, the number of rows of $P$ with at least one $\cross$, is maximized precisely when $\mathrm{pass}(P)$ is minimized.

  The affect of a simple transposition $s_i$ on the Lehmer code of $w$, assuming $\inv(s_i w) = \inv(w) - 1$, is to decrement $L(w)_{i}$ by $1$ and then swap this with $L(w)_{i+1}$. Therefore the minimum number of decreasing runs for a reduced decomposition of $w$ is $\max(L(w))$. When $\inv(s_i w) = \inv(w) - 1$, we necessarily have that $w_{i}>w_{i+1}$, and so $L(w)_{i} > L(w)_{i+1}$. Therefore this minimum is attained by the greedy bubble sort that begins by removing the rightmost descent of $w$ and continues by removing the next descent to the left of this until reaching the beginning of the word, then beginning again with the rightmost descent. Each pass decrements every positive value of the Lehmer code by $1$, so the number of passes is exactly $\max(L(w))$.


  If $j$ is the position of the first descent of $w$, then in any reduced decomposition for $w$, $s_i$ occurs to the left of some $s_{i+1}$ for any $i<j$. In particular, for any $P \in \QPD(w)$, if a row at or below row $j$ has no $\cross$, then all $\cross$'s occur strictly above that row. Let $k$ be the position of the largest number to the left of the first descent that is smaller than the smaller of the pair of entries involved in the first descent. Let $w^{\prime} = s_{k+1} \ldots s_{j-1} s_j w$, and let $\sigma^{\prime}$ be the greedy bubble sort expression for $w^{\prime}$. Then $\sigma = \sigma^{\prime} s_{k+1} \ldots s_{j-1} s_j$ is a reduced decomposition for $w$, and $\sit(\sigma)$ has its lowest $\cross$ exactly $\eta(w)$ rows below the $x$-axis. Note if all occurrences of $\max(L(w))$ are to the left of the first descent, then $\max(L(w'))=\max(L(w))-1$ and $\delta(w)=1$; otherwise $\max(L(w'))=\max(L(w))$ and $\delta(w)=0$.
 
  To see that no quasi-Yamanouchi pipe dream for $w$ has a $\cross$ in any row lower than this, note that if some row, say row $i$, has a $\cross$, but row $i+1$ does not, then there is a $\cross$ in the first column of row $i$, by the quasi-Yamanouchi condition, and this $\cross$ corresponds to the simple transposition $s_i$. Furthermore, the simple transposition corresponding to any $\cross$ above row $i$ necessarily has the form $s_k$ with $k>i+1$, and so $s_k$ and $s_i$ commute. Therefore we obtain another quasi-Yamanouchi pipe dream for $w$ by moving all $\cross$'s above row $i$ to weakly below it, corresponding to commuting all simple transpositions occuring to the right of the first $s_i$ to occur to the left of it. Furthermore, the lowest $\cross$ in the new pipe dream is at least as low as the lowest $\cross$ in the original pipe dream. Iterating this process as necessary, we may assume that there is a quasi-Yamanouchi pipe dream for $w$, say $Q$, with lowest $\cross$ as low as possible such that no row between the first $\cross$ and the last $\cross$ is empty. In particular, the lowest $\cross$ for $Q$ sits in the row $\mathrm{des}(Q)$ rows below the first $\cross$. If the first $\cross$ of $Q$ is in row $j$, then $\mathrm{des}(Q) \leq \mathrm{des}(\sit(\sigma))$, so the last $\cross$ of $Q$ is at or above that of $\sit(\sigma)$. Otherwise, the first $\cross$ of $Q$ is at least one row higher than the first $\cross$ of $\sit(\sigma)$ but $\mathrm{des}(Q) \leq \mathrm{des}(\sit(\sigma))+1$, so, again, the last $\cross$ of $Q$ is at or above that of $\sit(\sigma)$.
\end{proof}

\begin{theorem}
  The standardization map is well-defined and satisfies the following:
  \begin{enumerate}
  \item for $P \in \PD(w)$, $\stand(P) \in \RD(w)$;
  \item for $P \in \QPD(w)$, $\flatten(\wt(P))=\Des(\stand(P))$;
  \item the restriction $\stand:\QPD(w) \rightarrow \RD(w)$ is injective; and
  \item the restriction $\stand:\QPD(w) \rightarrow \RD(w)$ is surjective if and only if $\eta(w) \leq 0$.
  \end{enumerate}
  \label{thm:w-stand}
\end{theorem}

\begin{proof}
  Reading the $\cross$'s in the specified order always removes an adjacent inversion for $w$ proving (1). For (2), the quasi-Yamanouchi condition precisely gives that when reading the $\cross$'s left to right, top to bottom, each new (nonempty) row must begin with a lower index than the previous row ended with, and reading along a row increases indices. Therefore the descent composition is exactly the lengths of the nonempty rows. For (3), note that $\sit(\sigma)$ necessarily satisfies the quasi-Yamanouchi condition, and $\sit:\RD(w) \rightarrow \QPD(w)$ is a left inverse for standardization. Finally, (4) follows from Lemma~\ref{lem:lowest}, since there are no virtual quasi-Yamanouchi pipe dreams precisely when $\eta(w) \leq 0$.
\end{proof}

\begin{remark}
For $w$ Grassmannian, say $w = v(\lambda,n)$, we have $\inv(w) = |\lambda|$, $\max(L(w)) = \lambda_1$ with the unique maximum occuring at the unique descent, and $\min\{i \mid w_{i} > w_{i+1}\} = n$. Therefore $\eta(w) = |\lambda|-\lambda_1+1-n$. In particular, the standardization map on tableaux is surjective if and only if $n \geq |\lambda|-\lambda_1+1$.
\end{remark}

\begin{lemma}
  If $\sigma,\tau \in RD(w)$ and $\sigma$ differs from $\tau$ by a single commutativity relation or by a single braid relation, then the lowest cross of $\sit(\sigma)$ and the lowest cross of $\sit(\tau)$ lie in rows at most one apart.
  \label{lem:w-monotone}
\end{lemma}

\begin{proof}
Let $P=\sit(\sigma)$ and $Q=\sit(\tau)$ be the corresponding (virtual) pipe dreams. Suppose $\sigma$ and $\tau$ differ by a braid relation, say $\sigma =s_{i_{\ell}} \ldots s_{j+1} s_{j} s_{j+1} \ldots s_{i_1}$ and $\tau=s_{i_{\ell}} \ldots s_j s_{j+1} s_j\ldots s_{i_1} $. In $P$, the three $\cross$'s of the braid occupy all but the top-right corner of a $2\times 2$ block, and in $Q$ the braid $\cross$'s are all but the bottom-left corner of the $2\times 2$ block either one unit south or one unit west of that for $P$. If the block is one unit west in $Q$, there is no change between $P$ and $Q$ as to which rows contain $\cross$'s and which do not. 
The block is one unit south only when the block in $P$ is flush left. If there is no $\cross$ following the three braid $\cross$'s, then the lowest row of $Q$ is one row lower than that of $P$. Otherwise, $\sigma = s_{i_{\ell}} \ldots s_k s_{j+1} s_{j} s_{j+1} \ldots s_{i_1}$. If $k<j-1$, then the cross for $s_k$ is at least two rows below the braid $\cross$'s in $P$, and so it and all subsequent $\cross$'s are in the same row in both $P$ and $Q$. If $b\ge j-1$, then all $\cross$'s from this one up to and including the last $\cross$ before the highest empty row of $P$ (below the braid $\cross$'s) are one row lower in $Q$ than in $P$, but any further $\cross$'s are in the same row of both $P$ and $Q$. Thus if the highest empty row of $P$ (below the braid $\cross$'s) has a $\cross$ below it, then the lowest $\cross$ of $P$ and $Q$ are in the same row, otherwise the lowest $\cross$ of $Q$ is one row lower than the lowest $\cross$ of $P$.

Now assume $\sigma$ and $\tau$ differ by a commutativity relation, say $\sigma = s_{i_{\ell}}\cdots s_{d} s_{c} s_{b} s_{a} \cdots s_{i_1}$ and $\tau = s_{i_{\ell}} \cdots s_{d} s_{b} s_{c} s_{a} \cdots s_{i_1}$, with $b-c \geq 2$. Let $\cross_{x}$ be the $\cross$ corresponding to $x=a,b,c,d$. Up to and including $\cross_a$, $P$ and $Q$ are identical. We claim that the index of the row of $\cross_d$ differs by at most one between $P$ and $Q$. It follows that the lowest $\cross$ of $P$ and lowest $\cross$ of $Q$ lie at most one row apart: if $\cross_d$ is in the same row of both $P$ and $Q$, then the same is true for all $\cross$'s following $\cross_d$. If $\cross_d$ is one row lower in say $P$, then all $\cross$'s from $\cross_d$ up to and including the last $\cross$ before the highest empty row of $Q$ (below $\cross_d$) are also one row lower in $P$, but any further $\cross$'s are in the same row of both $P$ and $Q$. Thus if the highest empty row of $Q$ (below $\cross_d$) has a $\cross$ below it, then the lowest $\cross$ of $P$ and $Q$ are in the same row, otherwise the lowest $\cross$ of $P$ is one row lower than the lowest $\cross$ of $Q$.

To see the claim, we consider two subcases based on $a$. If (A1) $a>c$, then $\cross_c$ is in the same row in both $P$ and $Q$, with $\cross_b$ above $\cross_c$ in $P$ and to the right of $\cross_c$ in $Q$. If (A2) $c>a$, then $\cross_c$ is one row lower in $P$ than in $Q$, but still with $\cross_b$ above $\cross_c$ in $P$ and to the right of $\cross_c$ in $Q$. Now consider three subcases based on $d$. If (B1) $c>d$, then $\cross_d$ is in some row below $\cross_c$ of $P$ and some row below $\cross_b$ of $Q$. If (A1) holds then this is the same row in both $P$ and $Q$, while if (A2) holds this is either the same row in both $P$ and $Q$ or one row lower in $P$. If (B2) $b>d>c$, then $\cross_d$ is in the same row as $\cross_c$ in $P$, and one row lower than $\cross_b$ in $Q$. If (A1) holds, then $\cross_d$ is one row lower in $Q$ than in $P$, while if (A2) holds $\cross_d$ is in the same row of both $P$ and $Q$. If (B3) $d>b$, then $\cross_d$ is in the same row as $\cross_c$ in $P$ and the same row as $\cross_b$ in $Q$. If (A1) holds, then $\cross_d$ is in the same row in both $P$ and $Q$, while if (A2) holds $\cross_d$ is one row lower in $P$ than in $Q$.
\end{proof}

\begin{theorem}
  For $w$ a permutation, if $\eta(w) \leq 0$, then $^{\#}\QPD(w)= ^{\#} \RD(w)$, and otherwise
  \begin{equation}
    0 < ^{\#} \QPD(w) < \cdots < ^{\#} \QPD(1^{\eta(w)} \times w) = \cdots = ^{\#} \RD(w).
  \end{equation}
  \label{thm:w-monotone}
\end{theorem}

\begin{proof}
  Given $\sigma \in \RD(w)$, the position of the southernmost $\cross$ in $\sit(\sigma)$ precisely determines when $\sigma$ appears in the image of the standardization map for $1^m \times w$. Thus the theorem is equivalent to the statement that the rows of the southernmost $\cross$'s of the virtual pipe dreams corresponding to elements of $\RD(w)$ form an interval. Any element of $\RD(w)$ can be obtained from any other by a sequence of commutativity or braid relations. By Lemma~\ref{lem:w-monotone}, each step in the sequence changes the row of the lowest $\cross$ of the corresponding (virtual) pipe dream by at most one.
\end{proof}

For $m \geq \eta(w)$, the quasi-Yamanouchi pipe dreams are in bijection with reduced decompositions by Theorem~\ref{thm:w-stand}(3,4), and by Theorem~\ref{thm:w-stand}(2) the weights are the same. Thus we obtain our main result of this section, stating that, eventually and thereafter, the fundamental slide polynomial expansion of the Schubert polynomial flattens to the fundamental quasisymmetric expansion of the Stanley symmetric function. 

\begin{corollary}
  For any permutation $w$, let $\eta = \eta(w)$. Then, for any $m \geq \eta$, we have
  \begin{equation}
    \sch_{1^m \times w} = \sum_{a} [\Fund_{a} \mid \sch_{1^{\eta} \times w}]  \Fund_{0^{m-\eta}\times a} .
  \end{equation}
  In particular, taking the limit, we have
  \begin{equation}
    S_w=\lim_{m\rightarrow\infty} \sch_{1^m \times w} = \sum_{a} [\Fund_{a} \mid \sch_{1^{\eta} \times w}]  F_{\flatten(a)} (X).
  \end{equation}
  Moreover, this result is tight in the sense that if for some $n$ and for some $m>n$, we have
  \begin{equation}
    \sch_{1^m \times w} = \sum_{a} [\Fund_{a} \mid \sch_{1^{n} \times w}]  \Fund_{0^{m-n}\times a},
  \end{equation}
  then $n \geq \eta$.
  \label{cor:w-monotone}
\end{corollary}

For example, we have
\begin{eqnarray*}
  \sch_{24153} & = & \Fund_{(1,2,0,1)} + \Fund_{(2,1,0,1)} + \Fund_{(2,2,0,0)},  \\
  \sch_{135264} & = & \Fund_{(0,1,2,0,1)} + \Fund_{(0,2,1,0,1)} + \Fund_{(0,2,2,0,0)}+ \Fund_{(1,1,2,0,0)} + \Fund_{(1,2,1,0,0)}, \\
  \sch_{1246375} & = & \Fund_{(0,0,1,2,0,1)} + \Fund_{(0,0,2,1,0,1)} + \Fund_{(0,0,2,2,0,0)}+ \Fund_{(0,1,1,2,0,0)} + \Fund_{(0,1,2,1,0,0)}, \\
  & \vdots & \\
  \sta_{24153}(X) & = & F_{(1,2,1)}(X) + F_{(2,1,1)}(X) + F_{(2,2)}(X) + F_{(1,1,2)}(X) + F_{(1,2,1)}(X).
\end{eqnarray*}
Notice that $F_{(1,2,1)}(X)$ occurs with multiplicity $2$ in $\sta_{24153}(X)$ even though the expansions of the corresponding Schubert polynomials are multiplicity-free. One term appears immediately in $\sch_{24153}$, and the other first appears in $\sch_{1 \times 24153}$.

Combining Theorem~\ref{thm:stable-limit} and Theorem~\ref{thm:w-stand}(2), we obtain Theorem~\ref{thm:schub-limit} as a corollary.

%
\section{Structure constants}
%
\label{sec:struct}

\subsection{Quasi-slide product}
\label{sec:struct-quasi}

The utility of Schubert polynomials lies in the fact that they represent the Schubert classes of the flag variety, and so the structure constants of the Schubert polynomial basis enumerate points in generic triple intersections of Schubert subvarieties of the flag variety. To begin to understand these constants, we first give a combinatorial formula for the structure constants for slide polynomials, beginning with the monomial slide basis. This we do by generalizing the quasi-shuffle product of Hoffman \cite{Hof00}.

\begin{definition}[\cite{Hof00}]
  The \emph{quasi-shuffle product} of weak compositions $\alpha$ and $\beta$, denoted by $\alpha \qshuffle \beta$, is defined recursively by
  \begin{eqnarray*}
    \alpha \qshuffle \emptyset & = & \emptyset \qshuffle \alpha = \alpha, \\
    \alpha \qshuffle \beta & = & \alpha_1 (\alpha_2\cdots \alpha_{\ell(\alpha)} \qshuffle \beta) + \beta_1 (\alpha \qshuffle \beta_2\cdots \beta_{\ell(\beta)}) + [\alpha_1,\beta_1] (\alpha_2\cdots \alpha_{\ell(\alpha)} \qshuffle \beta_2\cdots \beta_{\ell(\beta)}),
  \end{eqnarray*}
  where $\emptyset$ is the empty composition, and $[\alpha_1,\beta_1]$ denotes the integer $\alpha_1+\beta_1$.
  \label{def:q-shuffle}
\end{definition}

For example, we have
\begin{displaymath}
  23 \qshuffle 11 = 2311 + 2131 + 2113 + 214 + 241 + 1231 + 1213 + 124 + 1123 + 331 + 313 + 34.
\end{displaymath}

\begin{remark}
  In what follows, we will assume weak compositions have the same length. If not, say $a$ has length $n$ and $b$ has length $m$, with $n>m$, then we may replace $b$ with $b\times 0^{n-m}$.
\end{remark}
  
\begin{definition}
  Let $a,b$ be weak compositions of length $n$. Let $\alpha = \flatten(a)$ and $\beta=\flatten(b)$. The \emph{quasi-shuffle set of $a$ and $b$}, denoted by $\QSS(a,b)$, is given by
  \begin{equation}
    \QSS(a,b) = \left\{ (\gamma_a, \gamma_b) \mid \begin{array}{ll}
      \flatten(\gamma_a)=\flatten(a), & \gamma_a \geq a, \\ 
      \flatten(\gamma_b)=\flatten(b), & \gamma_b \geq b,
    \end{array} \mbox{and } (\gamma_a+\gamma_b)_i>0 \mbox{ for all } i \right\}.
    \end{equation}
  \label{def:QSS}
\end{definition}

For example, writing $(\gamma_a,\gamma_b)$ as $\gamma_a+\gamma_b$, we have
\begin{displaymath}
  \QSS((0,2,0,3),(1,0,0,1)) = \left\{ \begin{array}{cc}
    (0,2,3)+(1,0,1) & (0,2,3)+(1,1,0) \\
    (2,0,3)+(1,1,0) & (2,3,0)+(1,0,1) \\
    (0,2,0,3)+(1,0,1,0) & (2,3)+(1,1) \\
    (0,2,3,0)+(1,0,0,1) & \\
  \end{array} \right\}.
\end{displaymath}

For a composition $c$ such that $\flatten(c) = \gamma_a+\gamma_b$, let $c = c_a + c_b$ be the unique decomposition such that $\flatten(c_a)=\gamma_a$ and $\flatten(c_b)=\gamma_b$.

\begin{definition}
  For weak compositions $a$ and $b$ of length $n$, define the \emph{quasi-slide product of $a$ and $b$}, denoted by $a \qslide b$, to be the formal sum of weak compositions defined by
  \begin{equation}
    a \qslide b = \sum_{(\gamma_a,\gamma_b) \in \QSS(a,b)} \bump_{(a,b)}(\gamma_a,\gamma_b),
  \end{equation}
  where $\bump_{(a,b)}(\gamma_a,\gamma_b)$ is the unique composition $c$ with $\flatten(c)= \gamma_a+\gamma_b$ such that $c_a \geq a$ and $c_b \geq b$ and if $\flatten(d) = \gamma_a+\gamma_b$ satisfies $d_a \geq a$ and $d_b \geq b$, then $d \geq c$.
  \label{def:quasi-slide}
\end{definition}

Continuing with our example, we have
\begin{eqnarray*}
  (0,2,0,3) \qslide (1,0,0,1) & = & (1,2,0,4) + (1,2,1,3) + (1,3,0,3) + (3,0,0,4) \\
  & & (3,0,1,3) + (1,2,3,1) + (3,0,3,1)
\end{eqnarray*}

The quasi-slide product is easily seen to be commutative and associative.

\begin{theorem}
  For weak compositions $a$ and $b$ of length $n$, we have
  \begin{equation}
    \Mono_{a} \Mono_{b} = \sum_{c} [c \mid a \qslide b] \Mono_{c},
  \end{equation}
  where $[c \mid a \qslide b]$ means the coefficient of $c$ in the quasi-slide product $a \qslide b$. 
  \label{thm:quasi-slide}
\end{theorem}

\begin{proof}
  From the definition of $\Mono_a$, we have $\Mono_{a} \Mono_{b} = \sum_{(a^{\prime},b^{\prime})} x^{a^{\prime} + b^{\prime}}$, where the sum is over all pairs $(a^{\prime},b^{\prime})$ such that $a^{\prime} \geq a$, $\flatten(a^{\prime})=\flatten(a)$, and $b^{\prime} \geq b$, $\flatten(b^{\prime})=\flatten(b)$. By taking $\bump_{(a,b)}(c)$ minimal, we collect together monomials occuring in a single monomial slide polynomial. 
\end{proof}

Using Theorem~\ref{thm:quasi-slide} and Theorem~\ref{thm:stable-limit} to take the stable limit, we obtain a result of Hoffman \cite{Hof00} that the quasi-shuffle product on (strong) compositions gives the structure constants for the monomial quasisymmetric functions.

\begin{corollary}[\cite{Hof00}]
  For (strong) compositions $\alpha$ and $\beta$, we have
  \begin{equation}
    M_{\alpha}(X) M_{\beta}(X) = \sum_{\gamma} [\gamma \mid \alpha \qshuffle \beta] M_{\gamma}(X),
    \label{e:quasi-shuffle}
  \end{equation}
  where $[\gamma \mid \alpha \qshuffle \beta]$ means the coefficient of $\gamma$ in the quasi-shuffle product $\alpha \qshuffle \beta$.
  \label{cor:quasi-shuffle}
\end{corollary}

\subsection{Slide product}
\label{sec:struct-slide}

We now give a combinatorial formula for the structure constants for fundamental slide polynomials by generalizing the shuffle product of Eilenberg and Mac Lane \cite{EM53} to weak compositions.

\begin{definition}[\cite{EM53}]
  The \emph{shuffle product} of words $A$ and $B$, denoted by $A \shuffle B$, is defined recursively by
  \begin{eqnarray*}
    A \shuffle \emptyset & = & \emptyset \shuffle A = \{A\}, \\
    A \shuffle B & = & \{A_1 (A_2 \cdots A_{\ell(A)} \shuffle B)\} \cup \{B_1 (A \shuffle B_2 \cdots B_{\ell(B)})\},
  \end{eqnarray*}
  where $\emptyset$ is the empty word. 
\label{def:shuffle}
\end{definition}

That is, $A \shuffle B$ is the set of all ways of riffle shuffling the terms of $A$, in order, with the terms of $B$, in order. For example, we have
\begin{eqnarray*}
  55111 \shuffle 82 = \left\{ \begin{array}{cccccc}
    5511182 & 5511812 & 5518112 & 5581112 & 5851112 & 8551112 \\
    5511821 & 5518121 & 5581121 & 5851121 & 8551121 & 5518211 \\
    5581211 & 5851211 & 8551211 & 5582111 & 5852111 & 8552111 \\
    5825111 & 8525111 & 8255111 & & &
  \end{array} \right\}.
\end{eqnarray*}

On the level of words, the quasi-shuffle product generalizes the shuffle product. However, the use of the two in giving rules for multiplying slide polynomials is far different. 

The \emph{descent composition of $C$}, denoted by $\Des(C)$, is the lengths of successive increasing runs of the letters read from left to right. For the example above, the last three terms on the right hand side have descent compositions $(2,2,3), (1,1,2,3), (1,3,3)$, respectively.

\begin{definition}
  Let $a,b$ be weak compositions of length $n$. Let $A$ and $B$ be the words defined by $A = (2n-1)^{a_1} \cdots (3)^{a_{n-1}} (1)^{a_n}$ and $B = (2n)^{b_1} \cdots (4)^{b_{n-1}} (2)^{b_n}$. Define the \emph{shuffle set of $a$ and $b$}, denoted by $\ShS(a,b)$, by
  \begin{equation}
    \ShS(a,b) = \{ C \in A \shuffle B \mid \Des_A(C) \geq a \mbox{ and } \Des_B(B) \geq b \},
  \end{equation}
  where $\Des_A(C)_i$ (respectively $\Des_B(C)_i$) is the number of letters from $A$ (respectively $B$) in the $i$th increasing run of $C$.
  \label{def:ShS}
\end{definition}

For example, $\ShS((0,2,0,3),(1,0,0,1))$ is given by
\begin{displaymath}
  \ShS((0,2,0,3),(1,0,0,1)) = \left\{ \begin{array}{ccccc}
    5581112 & 5851112 & 8551112 & 5581121 & 5851121 \\
    8551121 & 5581211 & 5851211 & 8551211 & 5582111 \\
    5852111 & 8552111 & 5825111 & 8255111 &
  \end{array} \right\}.
\end{displaymath}

\begin{definition}
  For weak compositions $a,b$ of length $n$, define the \emph{slide product of $a$ and $b$}, denoted by $a \slide b$, to be the formal sum
  \begin{equation}
    a \slide b =  \sum_{C \in \ShS(a,b)} \Des(\bump_{(a,b)}(C))
  \end{equation}
  where $\bump_{(a,b)}(C)$ is the unique element of $0^{n-\ell(\Des(C))} \shuffle C$ such that $\Des_A(\bump_{(a,b)}(C)) \geq a$ and $\Des_B(\bump_{(a,b)}(C)) \geq b$ and if $D \in 0^{n-\ell} \shuffle C$ satisfies $\Des_A(D) \geq a$ and $\Des_B(D) \geq b$, then $\Des(D) \geq \Des(\bump_{(a,b)}(C))$. 
  \label{def:slide}
\end{definition}

Continuing with our example, we have
\begin{eqnarray*}
  (0,2,0,3) \slide (1,0,0,1) & = & (3,0,0,4) + (2,1,0,4) + (1,2,0,4) + (3,0,3,1) + (2,1,3,1) \\
  & & (1,2,3,1) + (3,0,2,2) + (2,1,2,2) + (1,2,2,2) + (3,0,1,3) \\
  & & (2,1,1,3) + (1,2,1,3) + (2,2,0,3) + (1,3,0,3)
\end{eqnarray*}

Unlike the quasi-slide product, commutativity and associativity of the slide product is not immediate from the definition.

\begin{proposition}
  The slide product on weak compositions is commutative and associative.
\end{proposition}

\begin{proof}
  It suffices to show that in Definition~\ref{def:ShS}, for any $i=1,\ldots,n$, we may take $A^{\prime},B^{\prime}$ to be $A,B$, respectively, with the letters corresponding to $a_i,b_i$, say $2m-1$ and $2m$, interchanged without altering the slide product. This is trivial unless $a_i,b_i>0$. For $C \in A \shuffle B$, construct $C^{\prime}$ as follows. Mark every occurrence of $2m-1$ and $2m$ that occur in $C$ as $(2m)(2m-1)$. Unmarked occurrences must occur in strings of the form $(2m-1)^c(2m)^d$. Change each such string to $(2m-1)^d(2m)^c$, and call the resulting word $C^{\prime}$. Since the positions of descents are unchanged, we have $\Des_A(C) = \Des_{A^{\prime}}(C^{\prime})$ and $\Des_B(C) = \Des_{B^{\prime}}(C^{\prime})$, as required.
\end{proof}

Our main result of this section is that the slide product of compositions precisely gives the structure constants for the fundamental slide polynomials.

\begin{theorem}
  For weak compositions $a$ and $b$ of length $n$, we have
  \begin{equation}
    \Fund_{a} \Fund_{b} = \sum_{c} [c \mid a \slide b] \Fund_{c},
  \end{equation}
  where $[c \mid a \slide b]$ means the coefficient of $c$ in the slide product $a \slide b$. 
  \label{thm:slide}
\end{theorem}

\begin{proof}
  From the definition of $\Fund_a$, we have $\Fund_{a} \Fund_{b} = \sum_{(a^{\prime},b^{\prime})} x^{a^{\prime} + b^{\prime}}$, where the sum is over all pairs $(a^{\prime},b^{\prime})$ such that $a^{\prime} \geq a$, $\flatten(a^{\prime})$ refines $\flatten(a)$, and $b^{\prime} \geq b$, $\flatten(b^{\prime})$ refines $\flatten(b)$. By taking $\bump_{(a,b)}(C)$ maximal, we collect together monomials occuring in a single monomial slide polynomial just as in the quasi-slide product. Each part of $a$ and $b$ is represented by a different letter, with the letter for $a_i$ larger than that for $a_{i+1}$, and similarly for $b$. This ensures that taking $\Des_A$ of a shuffle of $A$ and $B$ will result in a refinement of $a$, and similarly for $b$. Finally, by taking the letter for $a_i$ larger than the letter for $b_i$, we ensure that each monomial slide polynomial occuring in the expansion of a fundamental slide polynomial on the right hand side is counted exactly once.
\end{proof}

We can use Theorem~\ref{thm:slide} together with Theorem~\ref{thm:stable-limit} to prove a result of Gessel \cite{Ges84}, stating that the structure constants for the fundamental quasisymmetric polynomials are given by the shuffle product of \emph{any} words representing the indexing compositions.

\begin{corollary}[\cite{Ges84}]
  For (strong) compositions $\alpha$ and $\beta$, we have
  \begin{equation}
    F_{\alpha}(X) F_{\beta}(X) = \sum_{C \in A \shuffle B} F_{\Des(C)}(X),
    \label{e:shuffle}
  \end{equation}
  where $A,B$ are any words with $\Des(A) = \alpha$, $\Des(B) = \beta$, and $A \cap B = \emptyset$, i.e. no letters appear in both $A$ and $B$.
  \label{cor:shuffle}
\end{corollary}

\begin{proof}
  If $\ell(\Des(A)) = \ell$, then replacing the letters of $A$, in order (e.g. by de-standardization), with any $\ell$-subset of positive integers clearly leaves the descent composition unchanged. Therefore we may assume that $A$ and $B$ use exactly $\ell(\alpha)$ and $\ell(\beta)$ letters, respectively. Given any choice of $A,B$, construct weak compositions $a,b$ of length $\ell(\alpha)+\ell(\beta)$ as follows. Assuming $\alpha_1,\ldots,\alpha_i$ and $\beta_1,\ldots\beta_j$ have been placed, if the letter corresponding to $\alpha_{i+1}$ is greater than the letter corresponding to $\beta_{j+1}$, then set $a_{i+j+1}=\alpha_{i+1}$ and $b_{i+j+1}=0$; otherwise set $a_{i+j+1}=0$ and $b_{i+j+1}=\beta_{j+1}$. By construction, $\flatten(a) = \alpha$ and $\flatten(b)=\beta$. By taking $m$ to be the length of the longest descent composition for any shuffle of $A \shuffle B$, we ensure $\ShS(0^m\times a,0^m\times b) = A \shuffle B$. The result now follows from Theorem~\ref{thm:slide} and Theorem~\ref{thm:stable-limit}.
\end{proof}

\subsection{Products of Schubert polynomials}
\label{sec:struct-schub}

Since the Schubert polynomial $\sch_{w}$ is a polynomial representative for the Schubert class of $w$ in the cohomology of the flag manifold, the coefficients $c_{u,v}^{w}$ defined by
\begin{equation}
  \sch_{u} \sch_{v} = \sum_{w} c_{u,v}^{w} \sch_{w},
\label{e:basis}
\end{equation}
enumerate flags in a generic triple intersection of Schubert varieties. Thus these so-called \emph{Littlewood--Richardson coefficients} are known to be nonnegative. A fundamental problem in Schubert calculus is to find a \emph{positive} combinatorial construction for $c_{u,v}^{w}$. One impediment to solving this problem is that computations quickly become intractable when multiplying out monomials. The following Littlewood--Richardson rule for the fundamental slide expansion of the product of Schubert polynomials gives us a more compact formula that should make computer experimentation possible.

\begin{theorem}
  For $u,v$ permutations and $a$ a weak composition, define $c_{u,v}^{a}$ by
  \begin{equation}
    \sch_{u} \sch_{v} = \sum_{a} c_{u,v}^{a} \Fund_{a}.
    \label{e:QLRR}
  \end{equation}
  Then we have
  \begin{equation}
    c_{u,v}^{a}  = \sum_{(P,Q) \in \QPD(u) \times \QPD(v)} [a \mid \wt(P) \slide \wt(Q)].
    \label{e:QLRR-formula}
  \end{equation}
  \label{thm:QLRR-sch}
\end{theorem}

\begin{proof}
  This follows from the characterization of the slide product in Theorem~\ref{thm:slide} and the fundamental slide expansion of Schubert polynomials in Theorem~\ref{thm:schubert-slide}.
\end{proof}

For example, we can compute the product $\sch_{24153} \sch_{2431}$ by
\begin{eqnarray*}
  \sch_{24153} \sch_{2431} & = & \left( \Fund_{(1,2,0,1)} + \Fund_{(2,1,0,1)} + \Fund_{(2,2,0,0)} \right) \left(\Fund_{(1,2,1,0)} + \Fund_{(2,1,1,0)} \right) \\
  & = &  \Fund_{(2,4,1,1)} + 2\Fund_{(3,3,1,1)} + \Fund_{(4,2,1,1)} + \Fund_{(2,4,2,0)} \\
  & & + 2\Fund_{(3,3,2,0)} + \Fund_{(4,2,2,0)} + \Fund_{(3,4,1,0)} + \Fund_{(4,3,1,0)} \\
  & = & \left( \Fund_{(2,4,1,1)} + \Fund_{(3,3,1,1)} + \Fund_{(4,2,1,1)} \right) + \left( \Fund_{(3,3,1,1)} \right) + \left( \Fund_{(3,3,2,0)} \right) \\
  & & + \left(\Fund_{(2,4,2,0)} + \Fund_{(3,3,2,0)} + \Fund_{(4,2,2,0)} \right) + \left( \Fund_{(3,4,1,0)} + \Fund_{(4,3,1,0)}\right) \\
  & = & \sch_{362415} + \sch_{45231} + \sch_{45312} + \sch_{364125} + \sch_{462135}.
\end{eqnarray*}
Here, in the last step we made use of the triangularity between the Schubert basis and the fundamental slide basis given in Proposition~\ref{prop:lead}.

In addition to improved computations, the product expansions for slide polynomials allow us to understand better the products of stable limits as well. To help analyze this stability, we define the following new statistic on pairs of (strong) compositions,
\begin{equation}
  \zeta(\alpha,\beta) = \min(|\alpha| + \ell(\beta), \ell(\alpha) + |\beta|).
  \label{e:zeta-strong}
\end{equation}
For example, $\zeta((2,3),(1,1)) = \min(5+2,2+2) = 4$.

\begin{lemma}
  Let $A,B$ be words with disjoint letters, and set $\alpha=\Des(A), \beta=\Des(B)$. Then there exists $C \in A \shuffle B$ such that $\ell(\Des(C)) = \zeta(\alpha,\beta)$, and for all $D \in A \shuffle B$, $\ell(\Des(D)) \leq \zeta(\alpha,\beta)$. 
  \label{lem:shuffle}
\end{lemma}

\begin{proof}
  By Corollary~\ref{cor:shuffle}, we may assume all letters in $A$ are smaller than all letters in $B$. Construct $C \in A \shuffle B$ using the greedy algorithm as follows. Assuming $C_1,\ldots,C_{h-1}$ have been chosen, say with $A_i\cdots A_{\ell}$ and $B_j \cdots B_m$ remaining, take $C_h$ to be the larger of $A_i,B_j$ that is smaller than $C_{h-1}$, or, if both are larger, take the larger of the two. This clearly maximizes the number of descents. 
\end{proof}

To extend $\zeta$ to pairs of weak compositions, let $|a| = \sum_{i} a_i$ and $\ell(a) = \ell(\flatten(a))$. Given a pair of weak compositions $(a,b)$, let $j_a$ be the smallest index such that $|a_1\cdots a_{j_a}| - \ell(a_1\cdots a_{j_a}) \geq |b| - \ell(b)$, and similarly define $j_b$. Let $1 \leq i_a < j_a$ (if $j_a$ is not defined, then $i_a$ ranges to $n$) be the index that maximizes $|a_1\cdots a_{i_a}|-i_a$. For example, if $a=(0,2,0,3)$ and $b = (1,0,0,1)$, then $j_a$ is undefined and $i_a = 1$. Note that, by construction, we always have $a_{i_a}, a_{j_a}>0$ when defined. Define $\zeta$ on weak compositions by
\begin{equation}
  \zeta(a,b) = \max \left(\begin{array}{cc}
    |a_1\cdots a_{i_a}|+\ell(b)-i_a-\epsilon(a_1\cdots a_{i_a},b), & \ell(a_1\cdots a_{j_a})+|b|-j_a, \\
    |b_1\cdots b_{i_b}|+\ell(a)-i_b-\epsilon(a,b_1\cdots b_{i_b}), & \ell(b_1\cdots b_{j_a})+|a|-j_b
  \end{array} \right),
  \label{e:zeta-weak}
\end{equation}
where $\epsilon(a,b)=1$ if there exists no $C \in \flatten(a)\shuffle\flatten(b)$ with a letter from $a$ appearing after the final descent of $C$, and  $\epsilon(a,b)=0$ otherwise. For example, $\zeta((0,2,0,3),(1,0,0,1)) = \max(2+2-1-2, 1+2-1-1) = 1$.

\begin{lemma}
  For weak compositions $a,b$, we have
  \begin{equation}
    0 < ^{\#}\ShS(a,b) < \cdots < ^{\#}\ShS(0^{\zeta(a,b)}\times a,0^{\zeta(a,b)}\times b) = \cdots = ^{\#}(A \shuffle B),
  \end{equation}
  where $A,B$ are any words with disjoint letters such that $\Des(A) = \flatten(a)$ and $\Des(B) = \flatten(b)$.
  \label{lem:product-ShS}
\end{lemma}

\begin{proof}
  Construct a word $Z \in A \shuffle B$  as in the proof of Lemma~\ref{lem:shuffle}, however, if, when doing this, taking $Z_h$ creates a descent with $Z_{h-1}$ and $\Des_{A}(Z_1\cdots Z_{h-1}) < a$ or $\Des_{B}(Z_1\cdots Z_{h-1}) < b$, then put $Z_h$ back and instead take all remaining letters equal to $A_i$, if the problem lies with $a$, and all remaining letters equal to $B_j$, if the problem lies with $b$, and put the smaller letters first. By construction, $Z \in \ShS(a,b)$, so $^{\#}\ShS(a,b)>0$.

  To see that all inequalities are strict, at the first time in the process that a problem occurs with $a$ or with $b$, we will construct an element $Z^{\prime}$ of $\ShS(0\times a,0\times b)$ that is not in $\ShS(a,b)$. If both $a$ and $b$ are problematic, say with $A_i < B_j$, then let $Z^{\prime}$ be the result of swapping the last occurence of $A_i$ with the $B_j$ that immediately follows it. If only one is problematic, say $a$, then the letter that the greedy algorithm first tried to take was $B_j$. Then let $Z^{\prime}$ be the result of moving $B_j$ to the left of the last $A_i$. The algorithm allows for only three possible cases: $B_j>A_i>Z_{h-1}$ or $Z_{h-1} > B_j>A_i$ or $A_i>Z_{h-1}>B_j$, and each is easy to see satisfies the claim.

  Finally, note that a problematic case never arises if and only if for every $i$ such that $a_i>0$ or $b_i>0$, the last occurence of the corresponding letter in \emph{any} shuffle $C \in A \shuffle B$ happens at or before the $i$th part of the descent composition of $C$. If no letter of $A$ occurs after the final descent in a word that maximizes the length of the descent composition, then we may slide the last letter coming from $A$ to end of the word, in the process losing one descent. Setting $\widehat{\zeta}(\alpha,\beta) = \zeta(\alpha,\beta) - \epsilon(\alpha,\beta)$, where $\epsilon$ is $0$ if $|\alpha|-\ell(\alpha) \geq |\beta|-\ell(\beta)$ and $-1$ otherwise, by  Lemma~\ref{lem:shuffle}, $\widehat{\zeta}(\Des(A),\Des(B))$ gives one plus the maximum number of descents that can occur before the last occurrence of a letter coming from $A$ in any shuffle of $A \shuffle B$. This means that in any shuffle of $A \shuffle B$, the last occurrence of the letter coming from $a_i$ occurs at or before position $\widehat{\zeta}(\flatten(a_1\cdots a_i), \flatten(b))$ in the descent composition. In order for the descent composition to dominate $a$, this position must be at or before $i$. Therefore $\widehat{\zeta}(\flatten(a_1\cdots a_i), \flatten(b)) - i$ precisely measures how many $0$'s must be prepended to $a$ to ensure $\sum_{j=1}^{i}\Des(C)_i \geq \sum_{j=1}^{i} a_i$ for all $C$. Finding this for each nonzero part of $a$ is equivalent to maximizing $\widehat{\zeta}(\flatten(a_1\cdots a_i), \flatten(b)) - i$ over all $i$.
  
  The expression $|a_1\cdots a_i|-\ell(a_1\cdots a_i)$ is monotonically increasing since the left term increases by at least one and the right by exactly one each time a nonzero $a_i$ is encountered. Therefore $\min(|a_1\cdots a_i| + \ell(b), \ell(a_1\cdots a_i) + |b|)$ occurs first at the left hand term, then at the right hand term, and never toggles back. Since $\ell(a_1\cdots a_i)$ increases by at most one as $i$ increases, the term $\ell(a_1\cdots a_i)+|b|-i$ is monotonically decreasing as $i$ increases. Therefore the maximum above is attained either at the index $i_a$ that maximizes $|a_1\cdots a_{i_a}|-i_a$, or at the first crossing point $j_a$ where $|a_1\cdots a_{j_a}| - \ell(a_1\cdots a_{j_a}) \geq |b| - \ell(b)$. The same analysis for $b$ results in \eqref{e:zeta-weak}.
\end{proof}

By Lemma~\ref{lem:product-ShS}, the product of fundamental slide polynomials stabilizes precisely at $\zeta(a,b)$. We can take this further by noting that fundamental expansion of the product of Schubert polynomials stabilizes precisely when both the individual expansions into fundamental slide polynomials and the product of those fundamental slide polynomials stabilize. To this end, extend the definition of $\zeta$ to pairs of permutations by
\begin{equation}
  \zeta(u,v) = \inv(u) + \inv(v) - \min(\width(v)+\min_{v_i > v_{i+1}}(i),\width(u)+\min_{u_i > u_{i+1}}(i)) + 1.
  \label{e:zeta-perm}
\end{equation}
For example, $\zeta(24153, 21534) = 4 + 3 - \min(2+2,3+1) + 1 = 4$.

\begin{theorem}
  For permutations $u,v$, let $\zeta = \zeta(u,v)$. Then for all $m \geq \zeta$, we have
  \begin{equation}
    \sch_{1^m \times u} \sch_{1^m \times v} = \sum_{a} c_{1^{\zeta}\times u,1^{\zeta}\times v}^{a} \Fund_{0^{m-\zeta}\times a},
  \end{equation}
  where $c_{u,v}^{a} = [\Fund_a \mid \sch_u \sch_v]$. In particular, taking the limit as $m\rightarrow\infty$, we have
  \begin{equation}
    \sta_{u} (X) \sta_{v} (X) = \sum_{a} c_{1^{\zeta}\times u,1^{\zeta}\times v}^{a} F_{\flatten(a)}(X).
  \end{equation}
  Futhermore, this result is tight in the sense that if $z$ is such that for some $m>z$, we have
  \begin{equation}
    \sch_{1^m \times u} \sch_{1^m \times v} = \sum_{a} c_{1^{z}\times u,1^{z}\times v}^{a} \Fund_{0^{m-z}\times a},
  \end{equation}
  then $z \geq \zeta$.
  \label{thm:product-ShS}
\end{theorem}

\begin{proof}
  By Corollary~\ref{cor:w-monotone}, the fundamental slide expansion of $\sch_{1^m\times u}$ is stable if and only if $m \geq \eta(u)$, and similarly for $v$. By Lemma~\ref{lem:product-ShS}, the fundamental slide expansion of the product $\Fund_{0^m\times a} \Fund_{0^m\times b}$ is stable if and only if $m \geq \zeta(a,b)$. Therefore the fundamental slide expansion of the product $\sch_{1^m\times u} \sch_{1^m\times v}$ is stable if and only if 
  \begin{equation}
    m \geq \eta + \max_{\substack{[\Fund_{a}\mid \sch_{1^{\eta}\times u}]>0 \\ [\Fund_{b}\mid \sch_{1^{\eta}\times v}]>0}} \left( \zeta(a,b) \right),
    \label{e:max}
  \end{equation}
  where $\eta=\max(\eta(u),\eta(v))$.Consider pairs $(a,b)$ that appear as $(a,b)=(\wt(P),\wt(Q))$ for some pair $(P,Q) \in \QPD(1^{\eta}\times u) \times \QPD(1^{\eta}\times v)$. We take each term of \eqref{e:zeta-weak} in turn. First, note that $|a| = \inv(u)$ and $|b| = \inv(v)$. Next, $j - \ell(a_1 \cdots a_j)$ is the number of empty rows up to and including row $j$ in $P$. Since the first $\eta - \eta(u)$ rows of $P$ are necessarily empty, we have $\ell(a_1 \cdots a_j) - j + |b| \leq \eta(u) - \eta + \inv(v)$. If $|a_1\cdots a_{i}| - \ell(a_1\cdots a_{i}) < |b| - \ell(b)$, then $|a_1\cdots a_{i}| + \ell(b) < |b| + \ell(a_1\cdots a_{i})$. Therefore $|a_1 \cdots a_i| + \ell(b) - i -\epsilon < |b| + \ell(a_1\cdots a_{i}) -i -\epsilon \leq \inv(v) + \eta(u) - \eta$. Combining these reductions with the symmetric ones with $a$ and $b$ interchanged, we have
  \begin{equation}
    \max_{\substack{[\Fund_{a}\mid \sch_{1^{\eta}\times u}]>0 \\ [\Fund_{b}\mid \sch_{1^{\eta}\times v}]>0}} \left( \zeta(a,b) \right) \leq \max( \inv(u) + \eta(v) - \eta, \inv(v) + \eta(u) -\eta ).
  \end{equation}
  Substituting this into \eqref{e:max} and expanding $\eta(u), \eta(v)$ gives the bound.
  
  To see that the bound is tight, let $P = \sit(\sigma_u)$ (resp. $Q = \sit(\sigma_v)$), where $\sigma_u$ (resp. $\sigma_v$) is the reduced decomposition for $u$ (resp. $v$) that spans $\eta(u)$ (resp. $\eta(v)$) contiguous rows, which exists by Lemma~\ref{lem:lowest}. Then $P$ (resp. $Q$) has exactly $\eta-\eta(u)$ (resp. $\eta-\eta(v)$) empty rows to begin.
\end{proof}

In \cite[Theorem 1.3]{Li14}, Li proved that the product $\sch_{1^m \times u} \sch_{1^m \times v}$ is stable for $m \ge \inv(u)+\inv(v)$. We can use Theorem~\ref{thm:product-ShS} to tighten the bound to $\zeta(u,v)$.

\begin{corollary}
  For permutations $u,v$, there exists $\zeta \leq \zeta(u,v)$ such that for all $m \geq \zeta$, we have
  \begin{equation}
    \sch_{1^m \times u} \sch_{1^m \times v} = \sum_{w} c_{1^{\zeta}\times u,1^{\zeta}\times v}^{w} \sch_{1^{m-\zeta}\times w},
  \end{equation}
  where $c_{u,v}^{w} = [\sch_w\mid\sch_u\sch_v]$. In particular, taking the limit as $m\rightarrow\infty$, we have
  \begin{equation}
    \sta_{u} (X) \sta_{v} (X) = \sum_{w} c_{1^{\zeta}\times u,1^{\zeta}\times v}^{w} \sta_{w}(X),
  \end{equation}
  giving the product of Stanley symmetric functions in terms of Schubert structure constants.
  \label{cor:product-ShS}
\end{corollary}

%
%

\bibliographystyle{amsalpha} 
\bibliography{quasi_schubert_rev1.bib}

\end{document}